\documentclass[11pt,hyp,english]{article}  

\usepackage{geometry}
\usepackage{babel}
\usepackage{amsmath,amssymb,amsthm,latexsym}
\usepackage{stmaryrd}
\usepackage{tikz}
\usepackage{tikz-cd}
\usepackage{array}
\usepackage{pdflscape}
\usepackage{booktabs}
\usepackage{mathtools}
\usepackage{tabularx}
\usepackage{cleveref}

\RequirePackage[backend    = biber,
                sortcites  = true,
                giveninits = true,
                doi        = false,
                isbn       = false,
                url        = false,
                style      = numeric]{biblatex}
\DeclareFieldFormat*{title}{#1}
\usepackage{csquotes}
\usepackage{dsfont}
\renewbibmacro{in:}{}

\RequirePackage[color       = blue!20,
                bordercolor = black,
                textsize    = tiny,
                figwidth    = 0.99\linewidth]{todonotes}

\newcommand{\A}{\mathbb{A}}
\newcommand{\D}{\mathsf{D}}
\newcommand{\E}{\mathsf{E}}
\newcommand{\F}{\mathsf{F}}
\newcommand{\Z}{\mathbb{Z}}
\newcommand{\N}{\mathbb{N}}
\newcommand{\Q}{\mathbb{Q}}
\newcommand{\FF}{\mathbb{F}}

\newcommand{\unit}{\mathds{1}}
\renewcommand{\P}{\mathbb{P}}
\newcommand{\HP}{\mathbb{HP}}
\newcommand{\SL}{\mathrm{SL}}
\renewcommand{\S}{\mathbf{S}}

\newcommand{\Gm}{\mathbb{G}_m}

\newcommand{\KGL}{\mathsf{KGL}}
\newcommand{\kgl}{\mathsf{kgl}}
\newcommand{\KQ}{\mathsf{KQ}}
\newcommand{\KW}{\mathsf{KW}}
\newcommand{\KO}{\mathsf{KO}}
\newcommand{\kq}{\mathsf{kq}}
\newcommand{\GW}{\mathsf{GW}}

\newcommand{\MZ}{\mathsf{M}\Z}
\newcommand{\MA}{\mathbf{M}A}
\newcommand{\MB}{\mathbf{M}B}

\newcommand{\HZ}{\mathsf{M}\Z}
\newcommand{\HQ}{\mathbf{M}\Q}
\newcommand{\cHZ}{\wt{{\mathsf{M}}}{\Z}}
\newcommand{\cHQ}{{\wt{\mathbf{M}}}{\Q}}

\newcommand{\HWZ}{\mathsf{W}\Z}

\newcommand{\R}{\mathbb{R}}
\newcommand{\SH}{\mathbf{SH}}
\newcommand{\Ab}{\mathbf{Ab}}

\newcommand{\Sm}{\operatorname{Sm}}
\newcommand{\Hom}{\operatorname{Hom}}
\newcommand{\Ext}{\operatorname{Ext}}

\newcommand{\holim}{\operatorname{holim}}

\newcommand{\tensor}{\otimes}
\newcommand{\Spec}{\operatorname{Spec}}
\newcommand{\Th}{\operatorname{Th}}

\newcommand{\Sq}{\operatorname{Sq}}
\newcommand{\wt}[1]{\widetilde{#1}}

\newcommand{\id}{\operatorname{id}}
\newcommand{\im}{\operatorname{im}}
\newcommand{\coker}{\operatorname{coker}}

\newcommand{\Char}{\operatorname{char}}
\newcommand{\pr}{\operatorname{pr}}
\newcommand{\Gr}{\operatorname{Gr}}
\newcommand{\HGr}{\mathrm{HGr}}

\newcommand{\inc}{\operatorname{inc}}
\newcommand{\s}{\mathsf{s}}

\newcommand{\vs}{\tilde{\mathsf{s}}}
\newcommand{\f}{\mathsf{f}}
\newcommand{\vf}{\tilde{{\mathsf{f}}}}
\newcommand{\cd}{{\mathsf{cd}}}

\newcommand{\MW}{{\mathsf{MW}}}

\newcommand{\sfM}{{\mathsf{M}}}

\newcommand{\vd}{{\tilde{\mathsf{d}}}}

\newcommand{\eff}{\text{eff}}
\newcommand{\veff}{\text{veff}}
\newcommand{\KMW}{\mathbf{K}^\mathrm{MW}}

\newcommand{\bideg}{(\star)}
\newcommand{\Forg}{\mathsf{Forg}}
\newcommand{\Hyper}{\mathsf{Hyp}}
\newcommand{\forg}{\mathsf{forg}}

\renewcommand{\sc}{\mathsf{sc}}

\let\iso=\cong
\let\smash=\wedge

\makeatletter
\def\iddots{\mathinner{\mkern1mu\raise\p@
    \vbox{\kern7\p@\hbox{.}}\mkern2mu
    \raise4\p@\hbox{.}\mkern2mu\raise7\p@\hbox{.}\mkern1mu}}
\makeatother

\author{H{\aa}kon Kolderup, Oliver R\"ondigs and Paul Arne {\O}stv{\ae}r}
\title{Hermitian $K$-theory and Milnor-Witt motivic cohomology over $\Z$}

\newtheorem{theorem}{Theorem}[section]
\newtheorem*{theorem*}{Theorem}
\newtheorem{lemma}[theorem]{Lemma}
\newtheorem{proposition}[theorem]{Proposition}
\newtheorem*{proposition*}{Proposition}
\newtheorem{corollary}[theorem]{Corollary}
\newtheorem*{corollary*}{Corollary}
\newtheorem{defn}[theorem]{Definition}
\newtheorem{defn*}{Definition}

\theoremstyle{definition}
\newtheorem{remark}[theorem]{Remark}
\newtheorem{remark*}{Remark}
\newtheorem{example}[theorem]{Example}

\begin{document}
\maketitle
\begin{abstract}
\noindent 
The theme of this paper is to compute hermitian $K$-groups in terms of 
the recently developed theory of Milnor-Witt motivic cohomology.
Our approach makes use of the very effective slice spectral sequence within the motivic stable homotopy category, 
which we analyze in detail for base schemes of arithmetic interest.
We show a Grothendieck-Riemann-Roch theorem, 
determine the map between Milnor-Witt and hermitian $K$-theory up to degree five for all fields, 
and compute the hermitian $K$-groups and the higher Witt-groups of the ring of integers.
\end{abstract}
\tableofcontents

\newpage
\section{Introduction}
\label{section:introduction}

The effective and very effective slice filtrations are computational tools in motivic homotopy theory with close connections to classical stable homotopy; see \cite{zbMATH07303324,levine.comparison,zbMATH02078162}. The Adams–Novikov spectral sequence determines the effective slices of the motivic sphere \cite[Conjecture 8.1]{levine.comparison}, \cite[Theorem 2.12]{zbMATH07003144}. Effective slices are accessible, but the very effective slices yield stronger invariants \cite{SO:twisted}. The recognition principle for infinite motivic loop spaces \cite{zbMATH07422194,zbMATH07733647} illustrates their strength. For the motivic sphere, only the zeroth and first very effective slices are known \cite[Theorem 3.30]{2020arXiv200506778B}.
The very effective slice spectral sequence sacrifices transparency for robustness: its input is less explicit, but its use is often simpler. The case of hermitian $K$-theory considered here builds on \cite{Bachmann} and simplifies parts of \cite{KRO}. Hermitian $K$-groups approximate the first stable homotopy groups of motivic spheres \cite{MorelICM06,zbMATH07003144,zbMATH07858209} and have been studied extensively in recent years \cite{9.I,9.II,9.III,dfjk,2019arXiv191111682H,zbMATH07623470}. Explicit motivic computations appear in \cite{berrick-karoubi,zbMATH05840069,zbMATH06261120,zbMATH05930233,zbMATH06433290,zbMATH07585647,zbMATH07352278,KRO,zbMATH07124513}.

\vspace{0.1in}

Let $\KQ$ denote the motivic spectrum representing Karoubi’s hermitian $K$-groups in the stable motivic homotopy category \cite{chn,hornbostel}. We compute hermitian $K$-groups via the very effective slice spectral sequence \cite{SO:twisted} (see \Cref{section:tvesss}):
\begin{equation}
\label{equation:vesssforKQ}
\pi_{\star}\vs_{\ast}\KQ \implies \pi_{\star}\KQ 
\end{equation}
The description of $\vs_{\ast}\KQ$ over fields of characteristic $\neq 2$ in \cite[Theorem 16]{Bachmann}, together with our extension to schemes essentially smooth over Dedekind schemes (\Cref{thm:veff-KQ}), identifies the $E_{1}$-page with Milnor–Witt motivic cohomology, a quadratic form analogue of Suslin–Voevodsky motivic cohomology and of Bloch’s higher Chow groups \cite{2020arXiv200406634B}.

The spectral sequence \eqref{equation:vesssforKQ} degenerates rationally. Completeness of the very effective slice filtration yields a Grothendieck–Riemann–Roch theorem for rational hermitian $K$-groups (cf.~\cite{Borel-iso}). The first nontrivial differential in \eqref{equation:vesssforKQ} is determined by motivic Steenrod operations, which allows identification of the image of Milnor–Witt $K$-theory in hermitian $K$-theory through degree five, extending \cite{KO-map}.

We also compute the algebraic $K$-theory, hermitian $K$-theory, and higher Witt groups of the ring of integers. The higher Witt groups follow the periodic pattern $\Z,\Z/2,0,0$. Certain algebraic and hermitian $K$-groups depend on the Kummer–Vandiver conjecture, a phenomenon first observed for $K(\Z)$ by Kurihara \cite{zbMATH00034713}.

\subsubsection*{Outline of the paper}

\Cref{aandh} establishes fundamental facts about algebraic and hermitian $K$-theory 
over qcqs (=quasi-compact and quasi-separated) base schemes.
The Wood motivic cofiber sequence relates $\KQ$ to algebraic $K$-theory $\KGL$ 
through the motivic Hopf map $\eta$.
\vspace{0.1in}

\Cref{section:tvesss} shows strong convergence for the very effective slice spectral 
sequence relative to any motivic spectrum over a qcqs base scheme of finite Krull dimension.
This is in contrast to the situation for the effective slice spectral sequence.
\vspace{0.1in}

We use the slice filtrations to define the Milnor-Witt motivic cohomology 
spectrum $\cHZ$ for schemes which are essentially smooth over a Dedekind scheme.
Due to \cite{Bachmann}, 
the very effective slice spectral sequence relates hermitian $K$-groups to the homotopy 
groups of $\cHZ$ for fields of characteristic unequal to $2$ in \cite{2020arXiv200406634B}.
\Cref{tvesohkt} deepens this result by extending it to base schemes of arithmetic interest, 
such as the integers $\mathbb{Z}$ and other rings of integers in number fields. 
\vspace{0.1in}

In \Cref{section:rdokt}, 
we work rationally and deduce rational splittings for both algebraic and hermitian $K$-theory.
Moreover, 
in \Cref{section:ldf}, 
we compute examples of hermitian $K$-groups.
These applications are elementary and practical in the sense that there are no obstructions 
in the form of nontrivial differentials in the very effective slice spectral sequence for $\KQ$.
\vspace{0.1in}

\Cref{section:tvezsokq} analyzes the very effective zero slice of $\KQ$.
With this in hand, 
we compute the first differentials in the very effective 
slice spectral sequence of $\KQ$ in \Cref{section:tfvesdfkq}.
As one would hope based on the proof of Milnor's conjecture on quadratic forms in \cite{RO:slices},
the very effective slice differentials can be expressed in terms of motivic Steenrod operations.
\vspace{0.1in}

In \Cref{section:tioMWKinhK}, 
we compute up to degree five the image of Milnor-Witt $K$-theory in 
hermitian $K$-theory for any field $F$. 
Our approach gives a quick proof of the isomorphism in degrees up to three 
established by Asok-Fasel \cite{KO-map}.
The calculations in degrees four and five are more interesting as they involve 
certain very effective slice differentials, which are shown to vanish by means of 
stable motivic cohomology operations and quadratic forms.
For example, 
we establish and use that for the integral motivic cohomology spectrum $\MZ$,  
the group $[\MZ,\Sigma^{3+(2)}\MZ]$ is cyclic of order six when viewed in the 
stable motivic homotopy category of the integers $\SH(\Z)$.
\vspace{0.1in}

In \Cref{section:hkgoti}, 
we explicitly compute the homotopy groups $\pi_{\star}\KQ$ over the integers $\mathbb{Z}$
using the very effective slice spectral sequence \eqref{equation:vesssforKQ}. 
As input, we use the Milnor and Bloch-Kato conjectures together with knowledge of 
$\pi_{\star}\vs_{\ast}\KQ$ and very effective slice differentials. 
We show that the very effective slice spectral sequence for $\KQ$ over $\mathbb{Z}$ 
collapses at its second page, 
and deduce the hermitian $K$-groups of $\mathbb{Z}$, as given in the table 
\begin{center}
 \begin{tabularx}{\textwidth}{c | c | c | c | l l l l} 
 \toprule
 $n>0$ 
 & $\KQ_{n,0}(\mathbb{Z})$ & $\KQ_{n+2,1}(\mathbb{Z})$ & $\KQ_{n+4,2}(\mathbb{Z})$ & $\KQ_{n+6,3}(\mathbb{Z})$ \\
 \midrule
$8k$ & $\mathbb{Z}\oplus\mathbb{Z}/2$ & $\Z/2\oplus K_n(\Z)$ & $0$  & $K_n(\Z)$ \\
$8k+1$ & $(\mathbb{Z}/2)^{3}$ & $\Z\oplus \Z/2$ & $0$ & $\Z^2\oplus \Z/2$ (for $n>1$)\\
$8k+2$ & $(\mathbb{Z}/2)^{2}\oplus K_{n}(\mathbb{Z})_{\text{odd}}$ & $0$ & $\Z\oplus K_{n}(\Z)_\mathrm{odd}$ & $(\Z/2)^2$\\
$8k+3$ & $\mathbb{Z}/w_{4k+2}$ & $\Z$ & $\Z/2w_{4k+2}$ & $(\Z/2)^2$ \\
$8k+4$ & $\mathbb{Z}$ & $K_n(\Z)$ & $\Z/2$ & $K_n(\Z)$ \\
$8k+5$ & $0$ & $\Z$ & $\Z/2$ & $\Z^2$\\
$8k+6$ & $K_{n}(\mathbb{Z})_{\text{odd}}$ & $\Z/2$& $\Z\oplus K_n(\Z)_{\mathrm{odd}}$ & $0$ \\
$8k+7$ & $\mathbb{Z}/w_{4k+4}$ & $\Z\oplus \Z/2$ & $\Z/w_{4k+4}$ & $0$ \\
\bottomrule
\end{tabularx}
\end{center}

We refer to \Cref{thm:coeff-KQZ} for a formulation using motivic cohomology groups instead of 
denominators of Bernoulli numbers and algebraic $K$-groups of the integers. 
The hermitian $K$-groups of the integers $\mathbb{Z}$ in weights $0$ and $2$ are in 
agreement with \cite[Theorem 3.2.1]{9.III}. 
(This reference does not discuss weights $1$ and $3$.)
The $(8,4)$-periodicity of $\KQ$ implies that it suffices to consider these four weights, 
but fewer weights do not give the complete picture.
The group $\pi_{7,3}\KQ(\Z)$ is infinite cyclic, 
a low-dimensional exception to the pronounced $8$-periodicity in the table.
Assuming the Kummer–Vandiver conjecture, 
stating that a prime \(p\) does not divide the class number of the maximal real subfield
of the \(p\)-th cyclotomic field,
the copies of $K_n(\Z)$ in the table for weight $1$ and $3$ are trivial 
\cite{zbMATH00034713}.
We deduce that the corresponding hermitian $K$-groups would satisfy an integral 
periodicity analogous to Bott's periodicity theorem in topological $K$-theory.

On the other hand, inverting the motivic Hopf map $\eta$ on $\KQ_\Z$ yields the higher Witt-theory 
spectrum $\KW_\Z$, whose coefficient ring, determined in \Cref{thm:coeff-KW}, admits the presentation
\[
  \pi_{\ast+\bideg}\KW_\Z = \Z[\alpha^{\pm 1},\eta^{\pm 1}, \xi]/(2\xi,\xi^2)
\]
Here $\alpha \in \pi_{4+(4)}\KW_\Z$ is a generator, and $\xi \in \pi_{1+(0)}\KW_\Z$ is the nonzero element. 
Thus inverting $\eta$ in hermitian $K$-theory leads to a simple periodic coefficient ring with polynomial and exterior parts (see below for our grading conventions).

\subsubsection*{Notation}
Let $S$ be a qcqs (quasi-compact and quasi-separated) base scheme.
For $t,w\in \Z$ we denote by
$S^{t,w}:=S^{t-w+(w)}:= S^{t-w}\wedge \mathbb{G}_{m}^{w}$ the motivic sphere of 
topological degree $t$ and weight $w$.
The corresponding suspension functor on the stable motivic homotopy category $\SH(S)$ 
is denoted by $\Sigma^{t,w}=\Sigma^{t-w+(w)}$.
To any motivic spectrum $\E\in\SH(S)$ the bigraded motivic homotopy group are defined by 
\[
\pi_{t,w}\E=
\pi_{t-w+(w)}\E=
[S^{t,w},\E]_{\SH(S)}
\]
In particular, 
the hermitian $K$-groups of $S$ can be defined as 
$\KQ_{t,w}=\pi_{t,w}\KQ$.
In the literature the notation $\GW^w_t$ is also used for hermitian $K$-groups. 
With our conventions one has
\[
\GW^{w}_{t}\cong [S^{t,0}, \Sigma^{(w)}\KQ]=\pi_{t-2w,-w}\KQ=\pi_{t-w-(w)}\KQ
\]
We write $\HZ$ (resp.\ $\HZ/2$) for the integral (resp.\ mod $2$) motivic cohomology spectrum, 
$\cHZ$ for the integral Milnor-Witt motivic cohomology spectrum, 
$\HWZ$ for the integral Witt motivic cohomology spectrum, $\unit$ for the motivic sphere spectrum,
and $\KMW$ for the Milnor-Witt ring; 
see \cite{2020arXiv200406634B,MorelICM06}.
When needed we indicate the base scheme in the notation, 
as in $\MZ_S$, suppressing ''$\Spec$'' for simplicity (e.g., $\KGL_\Z$).
For Milnor–Witt cohomology of fields we refer to \cite{deglise2023notesmilnorwittktheory}.
A Dedekind scheme $\mathcal{D}$ refers to a finite disjoint union of spectra of Dedekind 
domains or fields;
that is, a regular noetherian scheme of Krull dimension $\leq 1$.
We write $\mathcal{X}$ for an essentially smooth scheme over $\mathcal{D}$.

\subsection*{Acknowledgements}
This work was supported by the RCN Project no.~312472 Equations in Motivic Homotopy, the European Commission — Horizon-MSCA-PF-2022 Motivic integral $p$-adic cohomologies, the Deutsche Forschungsgemeinschaft (project numbers 426008713 and 539085450), and the Centre for Advanced Study at the Norwegian Academy of Science and Letters in Oslo. We thank Markus Spitzweck for valuable discussions and Jonas Kylling for contributions in the early stages of this project.

\section{Algebraic and hermitian $K$-theory}
\label{aandh}

Recall the motivic ring spectrum $\KGL$ in $\SH(\Z)$. 
It pulls back to any qcqs scheme and represents Quillen’s higher algebraic $K$-theory 
over regular base schemes \cite{ppr,rso.strict,voevodsky.icm}. 
It is periodic: there exists a Bott element
$\beta:\Sigma^{1+(1)}\unit\to \KGL$
such that multiplication by $\beta$ induces an equivalence $\KGL \simeq \Sigma^{1+(1)}\KGL$. Moreover, $\KGL$ carries a unique $E_\infty$-ring structure \cite{zbMATH06483840}.
Recent work has established a comparable status for Karoubi’s hermitian $K$-theory (higher Grothendieck–Witt theory), previously available only over bases in which $2$ is invertible \cite{hornbostel}.

\begin{theorem}[\cite{chn}]\label{thm:chn}
Let $S$ be a qcqs scheme. 
There exists a motivic $E_{\infty}$ ring spectrum $\KQ\in \SH(S)$ with the following properties.
\begin{enumerate}
    \item There exists an element $\alpha:\Sigma^{4+(4)}\unit\to \KQ$ such that multiplication with $\alpha$ is an equivalence $\KQ \simeq \Sigma^{4+(4)}\KQ$.
    \item\label{item:wood-seq} There exists a ring map $\Forg\colon \KQ_S\to \KGL_S$ and a $\KQ_S$-module map $\Hyper\colon \KGL_S\to \KQ_S$ such that \[ \Sigma^{(1)} \KQ \xrightarrow{\eta\smash \KQ} \KQ \xrightarrow{\Forg} \KGL \xrightarrow{\Sigma^{1+(1)}\Hyper \circ \beta} \Sigma^{1+(1)}\KQ \] is a cofiber sequence.
    \item For every morphism $f\colon R\to S$ there exists a canonical equivalence $f^\ast \KQ_R\simeq \KQ_R$.
    \item For every closed embedding $i\colon R\to S$ of regular schemes of codimension $c$ and normal bundle $Ni$, the purity transformation $i^\ast\KQ_S\to \Th(Ni)\smash i^!\KQ_S$ is an equivalence.
    \item For every vector bundle $V\to S$ of rank $r$, there exists a Thom equivalence $\KQ_S\simeq \Th(V)\smash \KQ_S(\det V[-r])$.\footnote{The shifted line bundle in parentheses denotes an appropriate twist of $\KQ_S$.}
    \item If $S$ is regular and $2$ is invertible on $S$, $\KQ_S$ represents Karoubi's hermitian $K$-groups over $S$.
\end{enumerate}
\end{theorem}

\Cref{thm:chn} does not permit transferring slice computations of hermitian $K$-theory from fields of characteristic zero to a more general base. To obtain an effectivity statement, additional geometric input is required. In his 2020 PhD thesis \cite{kumar.thesis}, supervised by the third author, K.~Arun Kumar constructs a hermitian $K$-theory spectrum $\KQ_S'$ over any qcqs scheme, providing the necessary geometric description; key results are published in \cite{kumar}. Let $\HGr_S$ denote the infinite quaternionic Grassmannian over $S$ from \cite{panin-walter.quat-borel}. Theorem~\ref{thm:chn-kumar} asserts that $\KQ_S'$ and $\KQ_S$ are equivalent as motivic spectra.

\begin{theorem}[\cite{kumar.thesis}]
Let $S$ be a qcqs scheme. 
There exists a motivic spectrum $\KQ_S^\prime\in \SH(S)$ with the following properties.
\begin{enumerate}
    \item[$3^\prime$.] For every morphism $f\colon R\to S$ there exists a canonical equivalence $f^\ast \KQ_R^\prime\simeq \KQ_R^\prime$.
    \item[$6^\prime$.] If $S$ is regular and $2$ is invertible on $S$, $\KQ_S$ represents Karoubi's hermitian $K$-groups over $S$.
    \item[$7^\prime$.] There exists a structure map $\Sigma^{8,4}\HGr_S\times \Z \to \HGr_S\times \Z$ of pointed motivic spaces over $S$ such that $\Sigma^{4,2}\KQ_S^\prime$ is the motivic spectrum associated to the Bousfield-Friedlander type $\Sigma^{8,4}$-spectrum $(\HGr_S\times \Z,\HGr_S\times \Z,\dotsc)$ obtained from this structure map.
\end{enumerate}
In particular, $\KQ^\prime_S$ is cellular \cite{rso.cell}.
\end{theorem}

We note that hermitian $K$-theory over schemes in which $2$ is not necessarily invertible has been considered previously \cite{schlichting,spitzweck.herm}, and the ring structure has been addressed in \cite{lopez.thesis}. The composition $\Hyper\circ \Forg$ coincides with the hyperbolic element $\mathsf{h}=1-\varepsilon$ on $\KQ$, whereas $\Forg \circ \Hyper$ coincides with the sum of the identity and the duality involution on $\KGL$. For a detailed discussion, see \cite[Section 3]{RO:slices}, which extends to the present setting.
For later use, we identify the $\eta$-inversion of hermitian $K$-theory with the Witt-theory spectrum $\KW$ \cite[Definition 8.1.1]{chn}.

\begin{proposition}
\label{prop.KQetaisKW}
The canonical map $\KQ[\eta^{-1}]\to \KW$ is an equivalence.
\end{proposition}
\begin{proof}
It suffices to prove the statement over $\Spec(\Z)$. 
We show that $\eta\smash \KW\colon \Sigma^{(1)}\KW\to \KW$ is an equivalence, 
so that the canonical map $\KQ \to \KW$ of motivic ring spectra induces a map $\KQ[\eta^{-1}]\to \KW$.
By appealing to the 
Tate sequence \cite[Corollary 8.1.6]{chn} in conjunction with the Wood sequence \cite[Corollary 8.1.7]{chn}, 
we need to identify the cofiber of the map $\eta\smash \KGL_{hC_2}$ between homotopy orbits with $\KGL$ 
under the composite of the hyperbolic and forgetful maps
\[
\KGL_{hC_2}\xrightarrow{\Hyper_{hC_2}} \KQ \xrightarrow{\Forg} \KGL
\]
This follows from the fact that the involution $\Psi^{-1}$ on $\KGL$ induces the involution $-\Psi^{-1}$ on $\Sigma^{(1)}\KGL$.
Given that $\eta\smash \KW$ is an equivalence, 
the canonical map $\KQ[\eta^{-1}]\to \KW$ of motivic ring spectra can be seen to be an equivalence 
by showing that $(\KGL_{hC_2})[\eta^{-1}]$ is contractible, or arguing via points (residue fields).
In the latter approach, 
the case of $\mathbb{F}_2$ uses \cite{chn}.
\end{proof}



As proven in \cite{kr.comparison}, the motivic spectra $\KQ_S$ and $\KQ_S^\prime$ are equivalent over any qcqs scheme $S$.
This adds a geometric description to every second infinite $\P^1$-loop space appearing in a 
motivic spectrum representing hermitian $K$-theory.

\begin{theorem}\label{thm:chn-kumar}
Let $S$ be a qcqs scheme.
There exists an equivalence $\KQ_S^\prime \to \KQ_S$.
Moreover, if $S$ is essentially smooth over a Dedekind scheme, $\Omega^{\infty+(\infty)}\Sigma^{2n+2+(2n+2)}\KQ_S$ is equivalent to $\Z\times \mathrm{HGr}_S$, 
where $\mathrm{HGr}_S$ is the infinite quaternionic Grassmannian over $S$.
\end{theorem}

\begin{proof}
    The equivalence is provided in \cite[Theorem 2.3]{kr.comparison}.
    The description of infinite $\P^1$-loop spaces follows from the case $S=\Spec(\Z)$ stated in \cite[Lemma 2.4]{kr.comparison}.
\end{proof}
    
\section{The very effective slice spectral sequence}
\label{section:tvesss}

Let $S$ be a qcqs base scheme of finite Krull dimension. The very effective slice filtration on $\SH(S)$ filters with respect to the projective line $\mathbb{P}^1$, or equivalently the Tate object $T=\mathbb{A}^1/(\mathbb{A}^1\setminus{0})$, rather than the multiplicative group $\mathbb{G}_m$ as in Voevodsky’s effective slice filtration \cite{zbMATH02078162}.
Following Spitzweck–{\O}stv{\ae}r \cite{SO:twisted}, let $\SH^\veff(S)$ denote the smallest full subcategory of $\SH(S)$ containing all suspension spectra of smooth finite type $S$-schemes and closed under extensions and homotopy colimits. Unlike Voevodsky’s effective subcategory $\SH(S)^\eff$, $\SH^\veff(S)$ is not triangulated. For $q\in\Z$, define
$\SH^\veff(S,q) =  \Sigma^{2q,q} \SH^\veff(S)$.
The full embedding of very effective motivic spectra participates in an adjunction
\begin{equation}
\label{equation:veffadjunction}
\tilde{\mathsf{i}}_{q} : \SH^\veff(S,q) \rightleftarrows \SH(S) : \tilde{\mathsf{r}}_{q}
\end{equation}
Setting $\vf_{q} = \tilde{\mathsf{i}}_{q}\circ \tilde{\mathsf{r}}_{q}$ produces natural cofiber sequences
\begin{align}
\label{veff-cof}
\vf_{q+1}\E  \to \vf_{q}\E  \to \vs_{q}\E 
\end{align}
The term $\vf_{q}\E $ is the $q$-th \emph{very effective cover of $\E$}
and $\vs_{q}\E $ is the $q$-th \emph{very effective slice of $\E$}. 
Applying homotopy to \eqref{veff-cof}, 
and letting $q$ vary, 
yields an exact couple and the associated very effective spectral sequence 
\begin{equation}
\label{eq:vsss}
\tilde{E}^1_{t,q,w} = 
\pi_{t,w}\vs_{q}\E 
\Longrightarrow 
\pi_{t,w}\E 
\end{equation}
The corresponding filtration on the target graded group $\pi_{t,w}\E $ is given by
\begin{equation}
\label{eq:vssstarget}
\vf_{q}\pi_{t,w}\E  
:= 
\im(\pi_{t,w}\vf_{q}\E  \to \pi_{t,w}\E)
\end{equation}
We note that the $r$-th differential takes the form $\tilde{E}^r_{p,q,w} \to \tilde{E}^r_{p-1,q+r,w}$.

Convergence is one of the merits of the very effective slice filtration
\begin{align}
\label{veff-cof2}
\cdots \to \vf_{q+1}\E  \to \vf_{q}\E  \to \vf_{q-1}\E  \to \cdots
\end{align}
By the construction in 
\cite[Section 5]{SO:twisted}, 
the connectivity of $\vf_{q}\E $ increases with $q$ so that the homotopy limit of 
\eqref{veff-cof2} is contractible,
\begin{equation}
\label{equation:veryeffectivelimit}
\holim_{q\to\infty}\vf_{q}\E \simeq\ast   
\end{equation}
The above informs us that \eqref{eq:vsss} is a conditionally convergent spectral sequence in the sense of 
\cite[Definition 5.10]{Boardman}.
It follows that the very effective slice filtration on motivic homotopy groups $\pi_{\ast,\ast}\E $ is complete;
see \cite[Lemma 5.11]{Boardman}.

\begin{proposition}
\label{proposition:strongconvergence}
The very effective slice spectral sequence \eqref{eq:vsss} is strongly convergent 
over qcqs base schemes of finite Krull dimension.
\end{proposition}
\begin{proof}
It remains to show that \eqref{veff-cof2} is exhaustive and Hausdorff;
see \cite[Lemma 3.1]{zbMATH06957413}. 
First we show that
\begin{align}
\label{eq:conv1}
\bigcup_{q\in\Z}\vf_{q}\pi_{t,w}\E  & = \pi_{t,w}\E 
\end{align}
The motivic sphere $S^{t,w}\in\SH^\veff(S,q)$ whenever $q\leq\min\{t-w,w\}$.
By abuse of notation, we do not distinguish $S^{t,w}$ as an object of $\SH^\veff(S,q)$ or $\SH(S)$.
Then, 
in the said range and with reference to the adjunction \eqref{equation:veffadjunction}, 
we have 
\[
[S^{t,w}, \E] \cong
[S^{t,w}, \tilde{\mathsf{r}}_{q}\E ] \cong
[S^{t,w}, \vf_{q}\E ]
\]
Hence any element of $\pi_{t,w}\E $ is in the image of $\pi_{t,w}\vf_{q}\E $ when $q\leq\min\{t-w,w\}$.

To verify the Hausdorff condition \cite[Definition 5.2(ii)]{Boardman} for \eqref{eq:vssstarget}, 
we consider an element
\begin{align}
\label{eq:conv2}
x\in \bigcap_{q\in\Z}\vf_{q}\pi_{t,w}\E  
\end{align}
We want to show that $x=0$.
By assumption $x\colon S^{t,w}\to\E$ factors through 
$\vf_{q}\E $ for all $q\in\Z$.
To conclude, 
it suffices to prove that if $\F\in\SH^\veff(S,q)$,
then every map $S^{t,w}\to\F$ is trivial 
for $q\gg 0$.
If the base scheme $S$ is $d$-dimensional, 
we show that the map is trivial for all $q\ge t+d-w$.
It suffices to consider $\F=T^{\wedge q}\wedge\F'$ 
for some $\F'\in\SH^\veff$.
Owing to 
\cite[Subsection 1.2.3, Theorem 7.15]{DKO}, 
there is an inclusion 
$\SH^\veff(S)\subseteq \SH^\eff_{\geq -d}(S)$.
Here, 
the homotopy $t$-structure on $\SH(S)$ is a special case 
of \cite[Proposition 1.4.4.11]{HA}
(and, 
moreover, 
for fields the inclusion is an equality).
Thus, 
since $S^{t-2q,w-q}=S^{t-q-w}\wedge \mathbb{G}_{m}^{w-q}$, 
the vanishing 
$$
[S^{t,w},T^{\wedge q}\wedge \F']\cong [S^{t-2q,w-q},\F']=0
$$ 
holds in the range $q\ge t+d-w$.
In summary, 
the very effective slice filtration \eqref{eq:vsss} is strongly convergent in the sense of 
\cite[Definition 5.2 (iii)]{Boardman}.
\end{proof}

\begin{remark}
The convergence of Voevodsky's slice spectral sequence is a more subtle problem;
see, e.g., 
\cite{BEO,zbMATH06220359,zbMATH07003144,zbMATH02078162} for the state of the art.
\end{remark}

\section{The very effective slices of hermitian $K$-theory}
\label{tvesohkt}
The very effective slices of $\KQ$ were computed in \cite[Theorem 16]{Bachmann} 
over fields of characteristic $\neq 2$, 
and recently extended to qcqs schemes in which $2$ is invertible 
\cite[Theorem 10.22]{bem.kglslices}. 
Using the extension of $\KQ$ from \cite{chn} to bases with residue fields of characteristic $2$, together with the results of Section~\ref{section:rdokt}, we generalize these identifications to schemes essentially smooth over a Dedekind scheme. Along the way, we also extend several related results, including the computation of effective slices of the very effective cover $\kq$ of $\KQ$ \cite{aro}.

\begin{theorem}
\label{BachmannKQ}
\label{thm:veff-KQ}
If the base scheme $\mathcal{X}$ is an essentially smooth scheme over a Dedekind scheme $\mathcal{D}$, 
the very effective slices of hermitian $K$-theory are given by 
\[
\vs_{q}\KQ
\simeq
\Sigma^{q+(q)}\begin{cases}
\vs_0\KQ & q \equiv 0 \bmod 4 \\
\HZ/2 & q \equiv 1 \bmod 4 \\
\HZ & q \equiv 2 \bmod 4 \\
0 & q \equiv 3 \bmod 4
\end{cases}
\]
\end{theorem}

More on the structure and properties of $\vs_0\KQ$ will follow below and in 
Section~\ref{section:tvezsokq}.
Our proof of \Cref{thm:veff-KQ} relies on \Cref{thm:chn-kumar} and 
\Cref{lem:low-degree-homotopy-sheaves-KQ} below.

\begin{lemma}\label{lem:low-degree-homotopy-sheaves-KQ}
Let $F$ be any field. 
\begin{enumerate}
    \item The forgetful map $\pi_{0+(0)}\Forg\colon \pi_{0+(0)}\KQ\to \pi_{0+(0)}\KGL = \Z$ is surjective.
    \item The forgetful map $\pi_{2+(2)}\Forg\colon \pi_{2+(2)}\KQ \to \pi_{2+(2)}\KGL=\Z$ coincides with the embedding of the even integers, and its cokernel is $\pi_{1+(1)}\KQ\cong \Z/2\Z$.
    \item The group $\pi_{-1+(-1)}\KQ$ is zero.
\end{enumerate}
\end{lemma}

\begin{proof}
The first statement follows from the fact that $\pi_{0+(0)}\KQ\to \pi_{0+(0)}\KGL = \Z$ 
is a unital ring map \cite{chn}.
The second statement can be seen as follows. 
The second assertion of Theorem~\ref{thm:chn-kumar} implies that $\pi_{2+(2)}\KQ$ is the group of 
$\A^1$-path components of $\Z\times \mathrm{HGr}$. 
Since $\mathrm{HGr}$ is $\A^1$-connected, this group is infinite cyclic.
A more precise identification goes as follows: 
Over any field, 
every nondegenerate symplectic form involves an even-dimensional vector space 
and admits a symplectic basis.
Thus, it would be more precise to identify $\pi_{2+(2)}\KQ$ with the group of even integers.
The forgetful map $\pi_{2+(2))}\KQ \to \pi_{2+(2)}\KGL=\Z$ coincides with the embedding 
of the even integers.
The Wood cofiber sequence \eqref{item:wood-seq} implies the remaining part of the second statement, 
once $\pi_{1+(2)}\KQ$ is shown to be zero.
In effect, we may use the Wood cofiber sequence again.
The vanishing $\pi_{2+(w)}\KGL=0=\pi_{1+(w)}\KGL$ for $w>2$ implies that 
$\pi_{1+(2)}\KQ\iso \pi_{1+(2)}\KQ[\eta^{-1}]$.
Since $\KQ[\eta^{-1}]\simeq \KW$ by~\Cref{prop.KQetaisKW}, the result follows from the vanishing of the odd homotopy groups of $\KW$ over fields of any characteristic; see \cite[Proposition 5.3.7]{chn}.
The third statement follows similarly, using part one, 
and the canonical isomorphism $\pi_{-1+(w)}\KQ\iso \pi_{-1+(w)}\KQ[\eta^{-1}]$ for $w>-2$. 
The latter is deducible via the Wood cofiber sequence again and the vanishing $\pi_{-1+(w)}\KGL=0$ for $w>-1$.
\end{proof}

\begin{proof}
All inputs in our proof of Theorem~\ref{thm:veff-KQ} are now in place.
As in \cite{Bachmann}, we begin with the second $\P^1$-slice.
To show that the map 
$\vs_0\Sigma^{2+(2)}\Hyper\colon \vs_0\Sigma^{2+(2)}\KGL\to \vs_0\Sigma^{2+(2)}\KQ$
is an equivalence, one may by essentially smooth base change assume that the base scheme is a Dedekind scheme $\mathcal{D}$.
Any point of $\mathcal{D}$ is either a closed point, corresponding to the closed embedding of the spectrum of a field, or the generic point, corresponding to an essentially smooth inclusion of the spectrum of a field.
Base change with respect to either of these morphisms preserves (very) effective covers, and hence (very) effective slices, by \cite[Lemma B.1, Lemma B.2]{bachmann-hoyois}.
By localization \cite[Prop.~B.3]{bachmann-hoyois}, it suffices to check that $\vs_0\Sigma^{2+(2)}\Hyper$ is an equivalence over fields.
Moreover, since any field is essentially smooth over a perfect subfield \cite[Lemma A.2]{hoyois.mgl-mz}, and the objects involved are compatible with essentially smooth base change, it suffices to consider the spectrum of a perfect field as base scheme.
The case of perfect fields of characteristic not two is shown in \cite[Lemma 13]{Bachmann}.
Using \Cref{thm:chn-kumar} and the second statement of \Cref{lem:low-degree-homotopy-sheaves-KQ}, the proof of \cite[Lemma 13]{Bachmann} extends to perfect fields of characteristic two, and hence supplies that $\vs_0\Sigma^{2+(2)}\Hyper$ is an equivalence. 
With this starting point and \Cref{lem:low-degree-homotopy-sheaves-KQ}, the proof of \cite[Theorem 16]{Bachmann} essentially transfers over. 
Again, it is convenient to express the statements in the form of a map.
The contractibility of the $\P^1$-slice $\vs_{-1}\KQ$ is equivalent to the canonical map $\vf_0\KQ\to \vf_{-1}\KQ$ being an equivalence.
The latter can be checked on perfect fields, so that the case of perfect fields of characteristic two can now be dealt with as in the proof of \cite[Theorem 16]{Bachmann}.
The first $\P^1$-slice can be deduced from 
\[ \vs_2\KQ\xrightarrow{\vs_2\Forg} \vs_2\KGL \to \vs_2\Sigma^{1+(1)}\KQ\simeq \Sigma^{1+(1)}\vs_1\KQ\]
This is a cofiber sequence, again by restricting to perfect fields and arguing as in the proof of \cite[Theorem 16]{Bachmann}.
Its first two terms have already been identified as $\Sigma^{2+(2)}\MZ$.
The group $[\MZ,\MZ]_{\SH(S)}$ is isomorphic to $[\unit,\MZ]_{\SH(S)}$ for $S$ essentially smooth over a Dedekind domain by 
\cite[Theorem B.5]{bachmann-hoyois} since $[\Sigma^n\f_1\unit,\s_0\unit]=0$ for all $n\in \Z$.
Hence by \Cref{lem:endo-mz-weight0}, the map $\vs_2\Forg$ corresponds uniquely to an integer over any field or Dedekind domain, which is invariant under base change.
Over fields of characteristic not two, this integer is $2$, as \cite{RO:slices} shows, and therefore also over all fields.
It follows that $\vs_1\KQ\simeq \Sigma^{1+(1)}\MZ/2$.
\end{proof}

The zeroth $\mathbb{P}^1$-slice of $\KQ$ is described via cofiber sequences in \Cref{section:tvezsokq}. Theorem~\ref{thm:veff-KQ} provides a key ingredient for a very effective version of the Wood cofiber sequence, which is of independent interest.

\begin{theorem}
\label{thm:woodkq}
If the base scheme $\mathcal{X}$ is an essentially smooth scheme over a Dedekind scheme $\mathcal{D}$, 
the very effective covers of algebraic and hermitian $K$-theory fit into the cofiber sequence
\[ 
\Sigma^{(1)}\kq \xrightarrow{\eta\smash \kq} 
\kq \xrightarrow{\vf_0\Forg} \kgl \to 
\Sigma^{1+(1)} \kq
\]
Moreover, 
the last map composes with the canonical map $\Sigma^{1+(1)}\kgl \simeq \vf_1\KGL \to \vf_0\KGL = \kgl$ to 
$\Sigma^{1+(1)}\vf_0\Hyper$.
\end{theorem}

\begin{proof}
Applying the functor $\vf_0$ to the Wood cofiber sequence
\[ \KQ \xrightarrow{\Forg}\KGL \to \Sigma^{1+(1)}\KQ\]
produces a sequence
\[ \kq \xrightarrow{\forg}\kgl \to \vf_0\Sigma^{1+(1)}\KQ\simeq \Sigma^{1+(1)}\vf_{-1}\KQ.\]
Since $\vf_0$ is not exact, one must argue to verify that the resulting sequence is indeed a cofiber sequence. By reducing to the case of a perfect base field, as in the proof of \Cref{thm:veff-KQ}, one may apply \cite[Lemmas 10 and 15]{Bachmann}, noting that the homomorphism
\[\pi_{0+(0)}\Sigma^{1+(1)}\Hyper\colon \pi_{0+(0})\Sigma^{1+(1)}\KGL \to \pi_{0+(0)}\Sigma^{1+(1)}\KQ\] 
is surjective; its target is trivial due to Lemma~\ref{lem:low-degree-homotopy-sheaves-KQ}.
It follows that the sequence above is a cofiber sequence.
Theorem~\ref{thm:veff-KQ} implies in particular that $\vs_{-1}\KQ$ is contractible, so the canonical map $\vf_0\KQ \to \vf_{-1}\KQ$ is an equivalence.
Hence, the sequence above has the form
\[ \kq \xrightarrow{\forg}\kgl\to \Sigma^{1+(1)}\kq \]
and the second map satisfies the stated compatibility with $\vf_0\Hyper$.
\end{proof}

The Wood cofiber sequence for $\kq$ and the slices of $\kgl$ (see \cite{levine.homotopy-coniveau} for fields, \cite{bachmann.jems} for Dedekind schemes, and \cite{bem.kglslices} for qcqs schemes) yields the slices of $\kq$, as in \cite[Theorem 17]{aro}.

\begin{corollary}\label{cor:slices-kq}
The nonnegative slices of $\kq$ over $\mathcal{X}$ are given by 
\[ \s_q\kq=
\begin{cases} 
\Sigma^{(q)} \bigl(\MZ/2\oplus  \Sigma^2\MZ/2 \oplus  \dotsm \oplus  \Sigma^{q-2} \MZ/2 \oplus  \Sigma^q \MZ\bigr) & q\equiv 0(2) \\
\Sigma^{(q)} \bigl(\MZ/2\oplus  \Sigma^2\MZ/2 \oplus  \dotsm \oplus  \Sigma^{q-2} \MZ/2 \oplus  \Sigma^q \MZ/2\bigr) & q\equiv 1(2)
\end{cases}
\]
The multiplicative structure is given by  
\[ \s_\ast \kq \simeq \MZ[\eta,\sqrt{\alpha}]/(2\eta=0, \eta^2\xrightarrow{\delta} \alpha) \]
Here $\eta$ has bidegree $(1)$ and $\sqrt{\alpha}$ has bidegree $2+(2)$.
The map $\Sigma^{(1)}\s_{n-1}\kq\simeq \s_{n}\Sigma^{(1)} \kq \xrightarrow{\s_{n}(\eta\smash \kq)}\s_n\kq$ is the canonical nontrivial one.
\end{corollary}

By \Cref{cor:slices-kq}, one obtains the effective slices of $\KQ=\kq[\alpha^{-1}]$, where $\alpha\in \pi_{4+(4)}\KQ$ denotes the Karoubi–Bott periodicity element, as well as of $\KW\simeq \KQ[\eta^{-1}]$. The latter will be relevant later, together with information about the first slice differential. Since $\kq$ is close to the motivic sphere spectrum $\unit$, its slices yield—albeit limited—information about the slices of $\unit$. This suffices, however, to determine all weight-$1$ endomorphisms of motivic Eilenberg–MacLane spectra; see \Cref{lem:endo-mz-weight1}. In turn, this constrains the possibilities for the first slice differential.

\begin{lemma}\label{lem:unit-kq-slices}
Let $\mathcal{X}$ be an essentially smooth scheme over a Dedekind scheme $\mathcal{D}$.
The unit map $\unit\to \kq$ induces equivalences on $\s_0$ and $\s_1$.
\end{lemma}

\begin{proof}
The case of $\s_0$ follows from \cite[Prop.~B.4]{bachmann-hoyois} and \Cref{cor:slices-kq}, together with the multiplicativity of the slice filtration \cite{grso}. 
The case of $\s_1$ can be seen by applying $\s_1$ to the map 
\[ \begin{tikzcd}
    \Sigma^{(1)}\unit \ar[r] \ar[d] &
    \unit \ar[r] \ar[d] &
    \unit/\eta \ar[d] \ar[r] &
    \Sigma^{1+(1)}\unit \ar[d] \\
    \Sigma^{(1)}\kq \ar[r] &
    \kq \ar[r] &
    \kgl \ar[r] &
    \Sigma^{1+(1)}\kq
\end{tikzcd}\] 
of cofiber sequences induced by the unit map $\unit\to \kq$.
The third vertical map corresponds to the standard inclusion 
\[ \P^2\hookrightarrow\P^\infty\hookrightarrow \Gr_\infty\] 
giving rise to the (very) effective cover $\kgl=\f_0\KGL$ over number rings by \cite{bachmann.jems}.
The cofiber of this inclusion can be determined by standard geometric arguments valid over any base scheme, 
whence the cofiber of $\unit/\eta \to \kgl$ is $2$-very effective.
The equivalence $\s_1\unit \simeq \s_1\kq$ can then be inferred from the diagram above and the equivalence 
for $\s_0$ already established.
\end{proof}

\begin{corollary}\label{cor:alpha1-powers}
Let $\mathcal{X}$ be as above and $0<n\in \N$.
Then $\s_n\unit_\mathcal{X}$ contains $\Sigma^{(n)}\MZ/2$ as a direct summand such that the Hopf map 
$\Sigma^{(1)}\s_n\unit_\mathcal{X}\to \s_{n+1}\unit_\mathcal{X}$ induces an equivalence on the respective direct summands.
\end{corollary}

\begin{proof}
This follows from \Cref{lem:unit-kq-slices}, together with the action of $\eta$ on the slices of $\kq$ given in \Cref{cor:slices-kq}.
\end{proof}

\begin{lemma}\label{lem:2-slice-sphere}
Let $\mathcal{X}$ be an essentially smooth scheme over a Dedekind scheme $\mathcal{D}$.
Then there is an equivalence $\s_2\unit_\mathcal{X}\simeq \Sigma^{(2)}\MZ/2 \oplus \Sigma^{1+(2)}\MZ/12$.
\end{lemma}

\begin{proof}
Using base change we may assume $\mathcal{X}=\Spec(\Z)$.
Let $\nu\colon \Sigma^{1+(2)}\to \unit$ denote the second motivic Hopf map obtained from the 
Hopf construction on $\SL_2\simeq \Sigma^{1+(2)}$.
Then the composite $\Sigma^{1+(2)}\xrightarrow{\nu}\unit\xrightarrow{\mathrm{unit}}\KQ$ is zero since 
$\pi_{1+(2)}\KQ=0$; see the proof of \Cref{lem:low-degree-homotopy-sheaves-KQ}.
Alternatively, 
$\nu$ is the attaching map for the top cell in the quaternionic projective plane $\HP^2$, 
whose hermitian $K$-theory splits by symplectic orientability.
Then also $\Sigma^{1+(2)}\xrightarrow{\nu} \unit\xrightarrow{\mathrm{unit}} \kq$ is zero, where the second map is the factorization of the unit map for $\KQ$, as its source is very effective.
Choosing the tautological extension $\unit/\nu\simeq\Sigma^{-2+(-2)}\HP^2 \to \kq$ provides a map which corresponds to the inclusion
\[ \HP^2\hookrightarrow \HP^\infty \hookrightarrow \HGr \]
in level two of the respective motivic spectra; see \cite[Theorem 2.6]{bachmann.jems}, 
which holds without any restriction on the characteristic.
As for the ordinary Grassmannian in the proof of \Cref{lem:unit-kq-slices}, 
geometric arguments from \cite{panin-walter.quat-borel} provide that the cofiber of the 
tautological extension $\unit/\nu\to \kq$ is $3$-effective.
Hence the cofiber sequence
\[ \Sigma^{1+(2)}\MZ\simeq \Sigma^{1+(2)}\s_0\unit \simeq \s_2\Sigma^{1+(2)}\xrightarrow{\s_2\nu}\s_2\unit \to \s_2\unit/\nu \simeq \s_2\kq 
\]
\[
\,\,\,\,\,\,\,\,\,\,\,\,\,\,\,
\simeq \Sigma^{(2)}\MZ/2\oplus \Sigma^{2+(2)}\MZ \to \Sigma^{2+(2)}\MZ \]
determines $\s_2\unit$.
By \Cref{lem:endo-mz-weight0}, the latter map is the projection onto the second summand followed by an integer. 
This integer does not change with the base.
For fields of characteristic zero, 
\cite[Corollary 2.13]{zbMATH07003144} shows that the said integer is $12$, giving the result.
\end{proof}

\begin{remark}\label{rem:slice-diff-numberring}
\Cref{lem:endo-mz-weight1} implies that for $n\in \N$, 
whenever the slices $\s_n\E$ and $\s_{n+1}\E$ of a motivic spectrum $\E\in \SH(\Z)$ 
such as $\unit,\kq$, or $\KGL$ are direct sums of motivic Eilenberg-MacLane spectra $\MA$ for finitely generated abelian groups $A$, 
the first slice differential $\s_n\to \Sigma\s_{n+1}\E$ coincides with the slice differential in $\SH(\Q)$.
\end{remark}

\begin{theorem}\label{thm:slices-KW}
Let $\mathcal{X}$ be an essentially smooth scheme over a Dedekind scheme $\mathcal{D}$. 
The slices of $\KW$ are given by 
\[ \s_\ast \KW \simeq 
\MZ[\eta^{\pm 1}\sqrt{\alpha}^{\pm 1}]/(2\eta=2\sqrt{\alpha}=0, \eta^2\xrightarrow{\delta} \sqrt{\alpha}) \]
The first slice differential has the following form on the summand $\Sigma^{2n+(m+2n)}\MZ/2$ indexed by the 
generator $\eta^m\sqrt{\alpha}^n$:
\[ \mathsf{d}^1(\eta^m\sqrt{\alpha}^n) = \begin{cases} 
\tau \eta^{m+3}\sqrt{\alpha}^{n-1} +(\Sq^2+\rho\Sq^1)\eta^{m+1}\sqrt{\alpha}^n 
+\Sq^3\Sq^1\eta^{m-1}\sqrt{\alpha}^{n+1} & n\equiv 1 \bmod 2 \\ \Sq^2\eta^{m+1}\sqrt{\alpha}^n 
+\Sq^3\Sq^1\eta^{m-1}\sqrt{\alpha}^{n+1} & n\equiv 0 \bmod 2\end{cases}\]
\end{theorem}

\begin{proof}
We deduce the slices from \Cref{cor:slices-kq}, 
the description of $\KW$ in \Cref{prop.KQetaisKW}, 
and the compatibility of slices with filtered colimits.
The differential follows from \cite{RO:slices} and \Cref{rem:slice-diff-numberring}.
\end{proof}

\begin{remark}
Heuristically, 
see \cite[Section 6]{grso} for details, 
the Betti realization of the very effective covers of $\KQ$ are the layers of the 
$S^{2}$-Postnikov tower of the real topological $K$-theory spectrum. 
\end{remark}


\begin{example}
\label{ex:veff-KQ}
Figure~\ref{fig:e1-KQ} displays, 
in terms of the data given by \Cref{thm:veff-KQ}, 
the first page of the very effective slice spectral sequence
\[ 
\pi_{s-(n)}\vs_{q}\KQ  
\Rightarrow 
\pi_{s-(n)}\KQ 
\]

\begin{figure}
\begin{center}
\resizebox{\linewidth}{!}{
\begin{tikzpicture}[font=\scriptsize,scale=1.2]
        \draw[help lines,xstep=2.0,ystep=1.0] (-0.5,-2.2) grid (8.9,4.5);
        \foreach \i in {-2,...,4} {\node[label=left:$\i$] at (-0.5,1.0*\i+.5) {};}
        \foreach \i in {0,...,4} {\node[label=below:$\i$] at (2.0*\i+.8,-2.4) {};}
        
{\node[above right=0pt] at (0.0,-2.0) {$H^{n-4,n-2}$};}
{\node[above right=0pt] at (0.0,-1.0) {$0$};}
{\node[above right=0pt] at (0.0,0.0) {$\widetilde{H}^{n,n}$};}
{\node[above right=0pt] at (0.0,1.0) {$0$};}
{\node[above right=0pt] at (0.0,2.0) {$0$};}
{\node[above right=0pt] at (0.0,3.0) {$0$};}
{\node[above right=0pt] at (0.0,4.0) {$0$};}

{\node[above right=0pt] at (2.0,-2.0) {$H^{n-5,n-2}$};}
{\node[above right=0pt] at (2.0,-1.0) {$0$};}
{\node[above right=0pt] at (2.0,0.0) {$\pi_{1-(n)}\vs_{0}(\KQ)$};}
{\node[above right=0pt] at (2.0,1.0) {${h}^{n+1,n+1}$};}
{\node[above right=0pt] at (2.0,2.0) {$0$};}
{\node[above right=0pt] at (2.0,3.0) {$0$};}
{\node[above right=0pt] at (2.0,4.0) {$0$};}

{\node[above right=0pt] at (4.0,-2.0) {$H^{n-6,n-2}$};}
{\node[above right=0pt] at (4.0,-1.0) {$0$};}
{\node[above right=0pt] at (4.0,0.0) {$\pi_{2-(n)}\vs_{0}(\KQ)$};}
{\node[above right=0pt] at (4.0,1.0) {${h}^{n,n+1}$};}
{\node[above right=0pt] at (4.0,2.0) {$H^{n+2,n+2}$};}
{\node[above right=0pt] at (4.0,3.0) {$0$};}
{\node[above right=0pt] at (4.0,4.0) {$0$};}

{\node[above right=0pt] at (6.0,-2.0) {$H^{n-7,n-2}$};}
{\node[above right=0pt] at (6.0,-1.0) {$0$};}
{\node[above right=0pt] at (6.0,0.0) {$\pi_{3-(n)}\vs_{0}(\KQ)$};}
{\node[above right=0pt] at (6.0,1.0) {${h}^{n-1,n+1}$};}
{\node[above right=0pt] at (6.0,2.0) {$H^{n+1,n+2}$};}
{\node[above right=0pt] at (6.0,3.0) {$0$};}
{\node[above right=0pt] at (6.0,4.0) {$0$};}

{\node[above right=0pt] at (8.0,-2.0) {$H^{n-8,n-2}$};}
{\node[above right=0pt] at (8.0,-1.0) {$0$};}
{\node[above right=0pt] at (8.0,0.0) {$\pi_{4-(n)}\vs_{0}(\KQ)$};}
{\node[above right=0pt] at (8.0,1.0) {${h}^{n-2,n+1}$};}
{\node[above right=0pt] at (8.0,2.0) {$H^{n,n+2}$};}
{\node[above right=0pt] at (8.0,3.0) {$0$};}
{\node[above right=0pt] at (8.0,4.0) {$\widetilde{H}^{n+4,n+4}$};}
\end{tikzpicture}
}
\end{center}
\caption{The terms $\pi_{s-(n)}\vs_{q}\KQ $ for fixed $n$, using $\pi_{0-(n)}\vs_{0}\KQ\cong \widetilde{H}^{n,n}$}
\label{fig:e1-KQ}
\end{figure}
\end{example}

To define the Milnor-Witt motivic cohomology spectrum for schemes which are 
essentially smooth over a Dedekind scheme, 
we first introduce the Witt motivic cohomology spectrum as the desuspension 
\begin{equation}
\label{eq:def-hwz}
\HWZ:= \Sigma^{(-1)}\f_1\vs_0\KQ
\end{equation}
As in \Cref{section:tvesss}, 
$\f_1$ denotes the first cover in Voevodsky's effective slice filtration, 
and $\vs_0\KQ$ is the zeroth very effective slice of hermitian $K$-theory. 
\Cref{prop:slices-cHZ} shows there is a canonical map $\HWZ\to\MZ/2$, 
which enables the next definition on account of 
Spitzweck's work on motivic cohomology in \cite{zbMATH07015021}.

\begin{defn}
\label{def:mwz}
The Milnor-Witt motivic spectrum $\cHZ$ over a scheme $\mathcal{X}$ essentially smooth over a Dedekind scheme is the homotopy pullback of 
effective motivic spectra
\[ 
\begin{tikzcd}
\cHZ \ar[r] \ar[d] & \HWZ \ar[d] \\
\MZ \ar[r] & \MZ/2
\end{tikzcd}
\]
in $\SH(\mathcal{X})$.
\end{defn}

If $2$ is invertible on $\mathcal{X}$, 
this is compatible with the definition given in \cite{bachmann.etaperiodic-dedekind}, 
as one may deduce from \Cref{thm:veff-KQ} and \cite[Theorem 10.22]{bem.kglslices}.
With rational coefficients, 
we recover the rationalized Milnor-Witt motivic cohomology 
spectrum $\cHQ$ in \cite[Def.~5.18, 6.1]{dfjk} over qcqs base schemes.
Of course, 
$\HWZ$ and $\cHZ$ can be extended to an arbitrary qcqs scheme $X$ by pullback 
along the structural morphism $X\to \Spec(\Z)$. 
\vspace{0.1in}

The very effective zeroth slice of $\KQ$ is closely related to two other motivic spectra of interest: Milnor–Witt motivic cohomology $\cHZ$ and Witt motivic cohomology $\HWZ$. The slices of these spectra were computed in \cite[§6]{Bachmann} over perfect fields of characteristic different from $2$, with the characteristic assumption entering explicitly in the proof of \cite[Lemma 20]{Bachmann}. To identify the first very effective slice differentials, and thereby enable more refined calculations, we require a slight strengthening involving the slices of $\vs_0\KQ$. Note first that $\vs_0\KQ$ fits into the homotopy cofiber sequence
\begin{equation}
\Sigma^{(1)}\HWZ \to \vs_0\KQ  \to \HZ \label{triangle2}
\end{equation}
where the second map is the canonical map
$\vs_0\KQ\simeq \vs_0\kq \to \s_0\vs_0\kq\simeq \s_0\kq\simeq \MZ$ involving Corollary~\ref{cor:slices-kq}.
The Witt cohomology spectrum $\HWZ:= \Sigma^{(-1)}\f_1\vs_0\KQ$ in \eqref{eq:def-hwz} is effective.
Using \Cref{prop:slices-cHZ} below, 
the canonical map $\HWZ\to \s_0\HWZ \simeq \MZ/2\oplus \Sigma^2\MZ/2$ 
composes with the projection to produce a map $\HWZ\to \MZ/2$.

\begin{proposition}
\label{prop:slices-cHZ}
The nontrivial slices of the effective motivic spectra $\HWZ$ and $\vs_{0}\KQ $ are 
\begin{align*}
\s_{q}(\HWZ) &\simeq \Sigma^{(q)}\HZ/2 \oplus \Sigma^{2+(q)}\HZ/2, q \geq 0 \\
\s_0(\vs_{0}\KQ ) & \simeq \HZ, \quad s_{q}(\vs_{0}\KQ ) \simeq \s_{q}(\HWZ), q > 0 
\end{align*}
The Hopf map $\eta$ induces an equivalence
\[  
\Sigma^{(1)}\s_{q}\vs_{0}\KQ \xrightarrow{\simeq} \s_{q+1}\vs_{0}\KQ  
\]
for $q>0$ and the canonical unique nontrivial map  
\[ 
\s_1(\eta)\colon \Sigma^{(1)}\s_0\vs_{0}\KQ \to \s_1\vs_{0}\KQ 
\] 
for $q=0$. For all $q\geq 0$, the map $\eta^q$ induces equivalences
$\Sigma^{(q)}\HWZ\xrightarrow{\simeq} \f_q\HWZ$ and 
$\Sigma^{(q)}\s_0\HWZ\xrightarrow{\simeq} \s_q\HWZ$.
\end{proposition}

\begin{proof}
The cofiber sequence
\[ \vf_1\kq \to \kq \to \vs_0\kq\]
induces a cofiber sequence
\begin{equation}\label{eq:1-sliceV} 
\s_1\vf_1\kq \to \s_1\kq\to \s_1\vs_0\kq 
\end{equation}
where the first term can be identified
as $\s_1\vf_1\kq\simeq \s_1\vs_1\kq\simeq \s_1\Sigma^{1+(1)}\MZ/2\simeq \Sigma^{1+(1)}\MZ/2$ using \Cref{thm:veff-KQ}, and the second term is $\Sigma^{(1)}\MZ/2$ by \Cref{cor:slices-kq}.
It follows that the first map in the cofiber sequence is up to suspension a map $\Sigma\MZ/2 \to \MZ/2$, which has to be zero by \Cref{lem:endo-mz-weight0}.
As a consequence, $\s_1\vs_0\KQ\simeq \Sigma^{(1)}\MZ/2\oplus  \Sigma^{3,1}\MZ/2$, in such a way that the Hopf map induces the canonical map
\[ \Sigma^{(1)}\MZ \simeq \Sigma^{(1)}\s_0\vs_0\KQ\simeq \s_1\Sigma^{(1)}\vs_0\KQ \xrightarrow{\s_1(\eta\smash \vs_0\KQ)} \s_1\vs_0\KQ\simeq \Sigma^{(1)}\MZ/2 \oplus  \Sigma^{3,1}\MZ/2.\]
To determine the slices $\s_n\vs_0\KQ$ for $n>1$, it remains to prove that the Hopf map induces an equivalence 
\[ \Sigma^{(1)}\s_{n-1}\vs_0\KQ\simeq \s_n\Sigma^{(1)}\vs_0\KQ \xrightarrow{\s_n(\eta\smash \vs_0\KQ)} \s_n\vs_0\KQ\]
for all $n>1$.
This is equivalent to the contractibility of $\s_n(\vs_0\kq/\eta)$ for all $n>1$.
Via the first cofiber sequence in this proof, this contractibility is equivalent to the canonically induced map
\[ \s_n\bigl((\vf_1\kq)/\eta\bigr) \to \s_n(\kq/\eta\simeq \kgl)\]
being an equivalence for all $n>1$, which in turn follows from 
\[ \f_2\bigl((\vf_1\kq)/\eta\bigr) \to \f_2(\kq/\eta\simeq \kgl)\]
being an equivalence.
It suffices to check this over perfect fields.
Applying the natural transformation $\vf_2\to \f_2$ to the cofiber sequence
\[
\begin{tikzcd}
\vf_1\kq\ar[r]  & 
(\vf_1\kq)/\eta\ar[r] & 
\Sigma^{1+(1)}\vf_1\kq
\end{tikzcd}
\]      
produces a commutative diagram
\[
\begin{tikzcd}
\vf_2\vf_1\kq\ar[r] \ar[d] & 
\vf_2\bigl((\vf_1\kq)/\eta\bigr)\ar[r] \ar[d] & 
\vf_2\Sigma^{1+(1)}\vf_1\kq\simeq \vf_2\Sigma^{1+(1)}\kq \ar[d] \\
\f_2\vf_1\kq \ar[r] &
\f_2\bigl((\vf_1\kq)/\eta\bigr) \ar[r] &
\f_2\Sigma^{1+(1)}\vf_1\kq\simeq \f_2\vf_2\Sigma^{1+(1)}\kq
\end{tikzcd}
\] 
in which the bottom row is a cofiber sequence.
The right-hand side vertical map is an equivalence because $\f_2\vf_2\simeq \vf_2$.
Via the equivalences $\vf_2\vf_1\simeq \vf_2\vf_2\simeq \f_2\vf_2$, Theorem~\ref{thm:veff-KQ} implies that the canonical map
\[ \f_2\vf_2\kq\to \f_2\vf_1\kq\] 
is an equivalence, as its cofiber $\f_2\Sigma^{1+(1)}\MZ/2 \simeq \ast$ is contractible.
Hence the left-hand side vertical map of the diagram above is also an equivalence. 
In particular, the outer terms in its lower row are $2$-very effective, whence so is lower middle term (and the upper middle term by construction).
Since the vertical map in the middle is universal with $2$-very effective source, it is also an equivalence, and the upper row is a cofiber sequence.
Applying the functor $\vf_2$ to the natural transformation 
\[
\begin{tikzcd}
\vf_1\kq\ar[r] \ar[d] & 
(\vf_1\kq)/\eta\ar[r] \ar[d] & 
\Sigma^{1+(1)}\vf_1\kq \ar[d] \\
\kq \ar[r] &
\kgl \ar[r] &
\Sigma^{1+(1)}\kq
\end{tikzcd}
\]      
of cofiber sequences yields a commutative diagram
\[
\begin{tikzcd}
\vf_2\vf_1\kq\ar[r] \ar[d] & 
\vf_2\bigl((\vf_1\kq)/\eta\bigr)\ar[r] \ar[d] & 
\vf_2\Sigma^{1+(1)}\vf_1\kq \ar[d] \\
\vf_2\kq \ar[r] &
\vf_2\kgl \ar[r] &
\vf_2\Sigma^{1+(1)}\kq
\end{tikzcd}
\] 
Here, the upper row is a cofiber sequence, as has just been established.
Also, 
the lower row is a cofiber sequence, 
by employing \cite[Lemma 15]{Bachmann} and the surjectivity of the homomorphism 
$\pi_{2+(2)}\KGL \to \pi_{2+(2)}\Sigma^{1+(1)}\KQ\iso \pi_{1+(1)}\KQ$ 
induced by the hyperbolic map; see \Cref{lem:low-degree-homotopy-sheaves-KQ}.
Since the outermost vertical maps are equivalences, so is the middle vertical map.
The $\P^1$- and $\Gm$-slice filtrations on $\KGL$ agree, whence $\vf_2\kgl \simeq \f_2\kgl$ canonically.
Hence $\vf_2\bigl((\vf_1 \kq)/\eta\bigr)\to \f_2(\kgl)$ is an equivalence.
As explained above, this concludes the description of the slices of $\vs_0\KQ$.

It remains to justify the last claim.
Since the canonical map $\f_2\bigl((\vf_1\KQ)/\eta\bigr)\to \f_2\bigl((\vf_0\KQ)/\eta\bigr)$ is an equivalence, its cofiber $\f_2\bigl((\vs_0\KQ)/\eta\bigr)$ is contractible.
Hence the canonical map $\f_2(\eta\smash \vs_0\KQ)\colon \f_2\Sigma^{(1)}\vs_0\KQ\to  \f_2\vs_0\KQ$ is an equivalence.
It follows that $\Sigma^{(1)}\f_1\vs_0\KQ\to \f_2\vs_0\KQ$ is an equivalence.
By the cofiber sequence \eqref{triangle2} and the vanishing $\f_1\MZ\simeq \ast$, we deduce that $\Sigma^{(1)}\f_1\Sigma^{(1)}\HWZ\to \f_2\Sigma^{(1)}\HWZ$ is an equivalence.
Since $\HWZ$ is effective, the natural compatibilities $\f_{n+1}\Sigma^{(1)}\simeq\Sigma^{(1)}\f_n$ for $n\in \{0,1\}$ imply the equivalence $\Sigma^{(1)}\HWZ\simeq \f_1\HWZ$, induced by $\eta$.
The claim then follows by induction.
\end{proof}

Recall from \eqref{def:mwz} that we define the effective motivic spectrum $\cHZ$ as the 
(homotopy) pullback
\[ 
\begin{tikzcd}
    \cHZ \ar[r] \ar[d] & \HWZ \ar[d] \\
    \MZ \ar[r,"\pr^\infty_2"] & \MZ/2
\end{tikzcd}
\]

\begin{proposition}\label{prop:chz}
There exists a unique map $\vs_0\KQ\to \cHZ$ such that its composition with the canonical map 
$\cHZ\to \MZ$ is the canonical map $\vs_0\KQ\to \s_0\vs_0\KQ\simeq \MZ$.
Moreover, there is a cofiber sequence
\begin{equation}
\Sigma^{1}\HZ/2 \to \vs_0\KQ  \to \cHZ \label{triangle}
\end{equation}
\end{proposition}

\begin{proof}
The defining (homotopy) pullback diagram for $\cHZ$ induces a long exact sequence
\[ \dotsm \to [\vs_0\KQ, \Sigma^{-1}\MZ/2] \to [\vs_0\KQ, \cHZ] \to [\vs_0\KQ, \MZ]\oplus [\vs_0\KQ, \HWZ] \to   [\vs_0\KQ, \MZ/2] \to \dotsm \]
The map $\eta\smash \vs_0\KQ\colon \vs_0\KQ \to \Sigma^{(-1)}\vs_0\KQ$ corresponds uniquely to a map $\vs_0\KQ\to \HWZ$.
This can be seen via the vanishing $[\vs_0\KQ,\Sigma^{-1+(-1)}\MZ] = [\vs_0\KQ,\Sigma^{(-1)}\MZ]=0$
in the long exact sequence
\[ \dotsm \to [\vs_0\KQ,\Sigma^{-1+(-1)}\MZ] \to [\vs_0\KQ,\HWZ]\to [\vs_0\KQ,\Sigma^{(-1)}\vs_0\KQ] \to [\vs_0\KQ,\Sigma^{(-1)}\MZ] \to \dotsm \]
induced by the $\Sigma^{(-1)}$-suspension of the cofiber sequence~\eqref{triangle2}.
The resulting composition $\vs_0\KQ\to \HWZ\to \MZ/2$ coincides with the composition $\vs_0\KQ \to \MZ \xrightarrow{\pr^\infty_2}\MZ/2$, as a straightforward computation provides.
Thus there exists a map $\vs_0\KQ \to \cHZ$ composing with the respective canonical maps to the "canonical" maps $\vs_0\KQ\to \MZ$ and $\eta\smash \vs_0\KQ$. 
Its uniqueness follows as soon as $[\vs_0\KQ,\Sigma^{-1}\MZ/2$ is the trivial group, 
which holds due to the Wood cofiber sequence for $\kq$ (\Cref{thm:woodkq}) and the vanishing $[\kgl,\Sigma^{-1}\MZ/2]=0$.
The latter vanishing can be deduced from the cofiber sequence
\[ \Sigma^{1+(1)}\kgl \simeq \f_1\kgl \to \kgl \to \s_0\kgl \simeq \MZ\]
together with \Cref{lem:endo-mz-weight0} and 
the fact that $1$-effective motivic spectra map trivially to $0$-slices.

It remains to identify the homotopy fiber $\F$ of the canonical map $\vs_0\KQ \to \cHZ$. This map factors through $\vs_0\KQ \to \MZ$, whose fiber is $\Sigma^{(1)}\HWZ$ by \eqref{triangle2}. Consequently, the cofiber of the induced map $\F \to \Sigma^{(1)}\HWZ$ is equivalent to the fiber of $\cHZ \to \MZ$. The latter is in turn equivalent to the fiber of $\HWZ \to \MZ/2$, which decomposes as $\f_1\HWZ \oplus \Sigma \MZ/2$. The claim follows since $\f_1\HWZ\simeq \Sigma^{(1)}\HWZ$ occurs as the first component of the second map in the cofiber sequence
\[ \F\to \Sigma^{(1)}\HWZ \to \f_1\HWZ\oplus \Sigma\MZ/2\]
\end{proof}

\section{Rational degeneration of $K$-theories}
\label{section:rdokt}

In this section, 
we study the very effective slice spectral sequence for hermitian $K$-theory with rational coefficients.
Let $\E\tensor \Q$ denote the rationalization of a motivic spectrum $\E$.

\begin{lemma}
\label{lemma:rationalizedsplices}
For every motivic spectrum $\E$ and integer $q$, 
the canonical map
\[
\vs_{q}(\E\tensor\Q)
\simeq
\vs_{q}(\E)\tensor\Q
\]
is an equivalence.
\end{lemma}
\begin{proof}
Use that $\SH^{\veff}(F)$ is closed under homotopy colimits and $\vf_q$, 
as a composition of a right and its left adjoint, 
preserves sequential colimits.
\end{proof}

\Cref{thm:veff-KQ}, \Cref{prop:chz}, and \Cref{lemma:rationalizedsplices}  
identify the very effective rationalized slices of $\KQ$.

\begin{corollary}
\label{corollary:rationalslices}
Let $\mathcal{X}$ be an essentially smooth scheme over a Dedekind scheme $\mathcal{D}$.
The very effective slices of $\KQ\tensor\Q$ are given by 
\[
\vs_{q}(\KQ\tensor\Q) 
\simeq
\Sigma^{2q,q}
\begin{cases}
\cHQ & q \equiv 0 \bmod 4 \\
\HQ  & q \equiv 2 \bmod 4 \\
0    & \text{else}
\end{cases}
\]
\end{corollary}

The homotopy pullback diagram \eqref{def:mwz} implies that $\cHQ\simeq \HQ \oplus \mathsf{W}\Q$, where the first summand is the motivic Eilenberg-MacLane spectrum with rational coefficients, and $\mathsf{W}\Q:=\HWZ\tensor \Q$.
This decomposition coincides with the orthogonal decomposition into the $+$-part and the $-$-part of $\SH(\mathcal{X};\Q)$.
Over Dedekind schemes in which $2$ is invertible, $\cHQ$ and $\mathsf{W}\Q$ coincide with the corresponding motivic spectra considered in \cite{dfjk}, 
since this is already true before passing to rational coefficients.
Moreover, the unit map $\unit \to \KQ$ induces an equivalence 
$\unit\tensor \Q\simeq \HQ\oplus \mathsf{W}\Q$ compatible under base change as in  
\cite[Corollary 6.2]{dfjk}.

\begin{lemma}
\label{lemma:Qd1zero}
In $\SH(\Z;\Q)$ the vanishing 
$ 
[\HQ, \Sigma^{t,w}\HQ] = 
0
$
holds for $w < 0$, and $t\neq 0=w$, and $t\neq 1$ for $w>0$.
\end{lemma}
\begin{proof}
This follows from the isomorphisms 
\[
[\HQ,\Sigma^{t,w}\HQ]_{\SH(\Z)}\cong [(\unit\tensor \Q)_+,\Sigma^{t,w}\HQ]_{\SH(\Z)}\cong H^{t,w}(\Z;\Q)
\]
We refer to \Cref{thm:HZ} for the groups $H^{t,w}(\Z;\Q)\cong H^{t,w}(\Z)\otimes\Q$.
\end{proof}

\begin{lemma}
\label{lemma:Qd1zero-part2}
In $\SH(\Z;\Q)$ the group $[\mathsf{W}\Q, \Sigma^{t,w}\mathsf{W}\Q]\cong 0$ is trivial for $t\neq w$.
\end{lemma}
\begin{proof}
This follows from the isomorphisms (the second one follows from \cite[Lemma 2.2]{dfjk})
\[
[\mathsf{W}\Q,\Sigma^{t,w}\mathsf{W}\Q]_{\SH(\Z)}\cong
[(\unit\otimes \Q)^-,\Sigma^{t,w}\mathsf{W}\Q]_{\SH(\Z)} \cong 
\]
\[
[(\unit\tensor \Q)^-,\Sigma^{t,w}\mathsf{W}\Q]_{\SH(\Q)} \cong 
\pi_{-t,-w}\mathsf{W}\Z_\Q[\eta^{-1}]\otimes \Q
\]
The latter group is concentrated on the line $t=w$ by \Cref{lem:pi0V}.
\end{proof}

\begin{proposition}
\label{proposition:vessscollapse}
Let $\mathcal{D}$ be a Dedekind scheme.
The very effective slice spectral sequence for $\KQ\tensor\Q$ collapses at its $\tilde{E}^1$-page.
\end{proposition}
\begin{proof}
We use induction to show the differentials are zero as maps of motivic spectra.
By \Cref{corollary:rationalslices} and base change, 
it suffices to work over $\Spec(\Z)$.
Since the very effective slices of $\KQ\tensor \Q$ are concentrated in even weights, the $\vd^1$-differentials are zero as maps of spectra.
This starts the induction. 
Let $q\in \Z$ be even.
If $\vd^n\colon \vs_q\KQ\tensor \Q \to \Sigma \vf_{q+n}\KQ\tensor \Q \to \Sigma\vs_{q+n}\KQ\tensor \Q$ is zero as a map of motivic spectra, the first map lifts to $\vf_{q+n+1}\KQ\tensor \Q$.
The resulting composition
$\vd^{n+1}\colon \vs_q\KQ\tensor \Q\to \vf_{q+n+1}\KQ\otimes \Q\to \vs_{q+n+1}\KQ\otimes \Q$ is then the $n+1$-st differential, as a map of motivic spectra.
If $n$ is even, the target is contractible.
If $n$ is odd, the $n+1$-st differential is, up to suspension, of the form
$\D\to\Sigma^{2n+3,n+1}\E$ for $\D,\E\in \{\cHQ,\HQ\}$ by \Cref{corollary:rationalslices}.
Thus, by \Cref{lemma:Qd1zero,lemma:Qd1zero-part2}, $\vd^{n+1}$ is trivial as a map of motivic spectra.
This concludes the induction step.
\end{proof}

Next we compute the rational hermitian $K$-groups of smooth schemes of finite type over Dedekind schemes, 
cf.~the works by Ananyevskiy-Levine-Panin \cite[Corollary 3.5]{ALP} and 
D{\'e}glise-Fasel-Jin-Khan \cite[Theorem 2.3]{Borel-iso}.

\begin{corollary}
\label{corollary:rationalhermitianKgroups}
Let $\mathcal{D}$ be a Dedekind scheme.
For every $X\in \Sm_{\mathcal{D}}$ and integers $t$, $w$, there are isomorphisms
\[
\KQ_{t,w}(X,\Q) 
\cong 
\bigoplus_{q \equiv 0 \bmod 2} \widetilde{H}^{4q-t,2q-w}(X,\Q) 
\oplus 
\bigoplus_{q \equiv 1 \bmod 2} H^{4q-t,2q-w}(X,\Q)
\]
\end{corollary}
\begin{proof}
The very effective slice spectral sequence for $\KQ\tensor\Q$ takes the form
\[
[\Sigma^{t,w}\Sigma^\infty X_+, \vs_{q}(\KQ\tensor\Q)] \implies 
[\Sigma^{t,w}\Sigma^\infty X_+, \KQ\tensor{\Q}] = \KQ_{t,w}(X,\Q)
\]
\Cref{proposition:vessscollapse} implies that it collapses, and \Cref{corollary:rationalslices} supplies the answer.
\end{proof}

Our next result gives a Grothendieck-Riemann-Roch theorem for hermitian $K$-theory.

\begin{theorem}
\label{theorem:eoosplitting}
Over a Dedekind scheme the very effective slice filtration induces an equivalence of motivic spectra
$$
\KQ\tensor\Q
\simeq
\bigoplus_{q\in \Z}\Sigma^{8q,4q}\cHQ \oplus  \bigoplus_{q\in \Z}\Sigma^{8q+4,4q+2}\HQ
$$
\end{theorem}
\begin{proof}
For any motivic spectrum $\E$ there exist natural maps 
\begin{equation}
\label{equation:twomaps}
\E \leftarrow  \bigoplus_{q\in \Z} \vf_q\E \to \bigoplus_{q\in \Z} \vs_q \E 
\end{equation}
which, if $\E$ is an $E_\infty$-ring spectrum, are maps of such, using the multiplicative properties of the very effective slice filtration recorded in 
\cite[Sections 5,6]{grso} and \cite[Theorems 3.8, 3.14]{grso-2}.
The connecting map 
\[ \bigoplus_{q\in \Z}\vs_q\KQ\tensor \Q\to \bigoplus_{q\in \Z}\Sigma \vf_{q+1}\KQ\tensor \Q\]
is zero by combining \Cref{proposition:vessscollapse} and 
completeness of the very effective slice filtration; 
$\holim_{q\to\infty}\vf_{q}\KQ\tensor \Q \simeq\ast$ as noted in \eqref{equation:veryeffectivelimit}.
We obtain a splitting of the rightmost map in \eqref{equation:twomaps}, 
i.e., 
$\bigoplus_{q\in \Z} \vs_q \KQ\tensor \Q$ is a direct summand of 
$\bigoplus_{q\in \Z} \vf_q \KQ\tensor \Q$.
We claim the induced map 
\begin{equation}
\label{equation:rationalequivalence}
\bigoplus_{q\in \Z} \vs_q \KQ\tensor \Q\to \KQ\tensor \Q
\end{equation}
is an equivalence. 
This follows from \Cref{proposition:strongconvergence} since applying the very effective 
slice functor $\vs_q$ to \eqref{equation:rationalequivalence} yields an equivalence for 
every $q\in\Z$.
\Cref{corollary:rationalslices} finishes the proof.
\end{proof}

From this proof it is not clear whether the equivalence in \Cref{theorem:eoosplitting} respects the $E_\infty$-structures.
Over a qcqs scheme on which $2$ is invertible, this follows from \cite{deglise-fasel.borel}.
The above discussion also applies to algebraic $K$-theory and yields an alternative proof of Bloch's version 
of the Grothendieck-Riemann-Roch theorem relating rational algebraic $K$-groups and higher Chow groups 
\cite{zbMATH03983347}.

\begin{theorem}
\label{theorem:eoosplittingkgl}
Let $\mathcal{D}$ be a Dedekind scheme. 
There is an equivalence of motivic spectra
$$
\KGL\tensor\Q
\simeq
\bigoplus_{q\in\Z}\Sigma^{2q,q}\HQ
$$
Moreover, for every $X\in \Sm_{\mathcal{D}}$ and integers $t$, $w$, there are isomorphisms
\[
\KGL_{t,w}(X,\Q) 
\cong 
\bigoplus_{q\in\Z} H^{2q-t,q-w}(X,\Q)
\]
\end{theorem}

\section{Hermitian $K$-groups of fields}
\label{section:ldf}

In this section, we compute examples of hermitian $K$-groups in terms of Milnor-Witt motivic 
cohomology groups over fields, extending \Cref{lem:low-degree-homotopy-sheaves-KQ}. 

\begin{theorem}\label{thm:KQ-coeff-field}
Let $F$ be a field. 
There are isomorphisms 
\begin{itemize}
\item[0.] 
$\KQ_{n,n}(F)\cong \KMW_{-n}(F) 
\cong
\widetilde{H}^{-n,-n}(F)
\cong
\begin{cases}
\GW(F) & n=0  \\
\mathsf{W}(F) & n>0
\end{cases}
$ 
\item[1.] 
$
\KQ_{1,0}(F)\cong 
h^{1,1}(F)\oplus h^{0,1}(F)
$, 
$
\KQ_{2,1}(F)\cong 
h^{0,0}(F)
$,
$
\KQ_{1+n,n}(F)=
0
$
for $n>1$.
\item[2.] 
$
\KQ_{2,0}(F)\cong 
H^{2,2}(F)\oplus h^{0,1}(F)
\cong
K_{2}(F) \oplus \mathbb{Z}/2
$, 
$
\KQ_{3,1}(F)\cong 
H^{1,1}(F)
$,
$
\KQ_{4,2}(F)\cong 
H^{0,0}(F)
$,
$
\KQ_{2+n,n}(F)=
0
$
for $n>2$.
\item[3.] 
$
\KQ_{3,0}(F)\cong 
H^{1,2}(F)
$,
$
\KQ_{3+n,n}(F)=
0
$
for $n>0$, and an exact sequence
$$
0\to
H^{2,3}(F)
\to
\KQ_{2,-1}(F)
\to
h^{0,2}(F)
\to 0
$$
\end{itemize}
\end{theorem}
\begin{proof}
We can read off most of these calculations from \Cref{fig:e1-KQ} for a fixed $n$. 
The zeroth part, when $n\geq 0$ and $s=0$, is immediate since $\widetilde{H}^{-n,-n}(F)$ is the only term contributing to $\KQ_{n,n}(F)$.
The identification of $\KQ_{0,0}(F)$ with the Grothendieck-Witt ring $\GW(F)$ and of $\KQ_{n,n}(F)$ for $n>0$ with the Witt ring $\mathsf{W}(F)$ of $F$ follows from \cite[Cor.~1.3.10]{9.III}.
Note furthermore that \cite{deglise2023notesmilnorwittktheory} identifies the (Grothendieck-)Witt ring of $F$ with the corresponding Milnor-Witt $K$-group of $F$ in any characteristic.
For the first part, 
when $n=0$ and $s=1$, 
we note there is a short exact sequence
$$
0 \to
h^{1,1}(F) \to
\KQ_{1,0}(F) \to
\pi_{1,0}\vs_{q}\KQ  \to
0
$$
The sequence \eqref{triangle} implies 
$\pi_{1,0}\vs_{q}\KQ \cong \pi_{1,0}\Sigma^{1}\HZ/2 \cong h^{0,0}(F)$ 
because $\pi_{s,0}\cHZ\cong \pi_{s,0}\HZ \cong h^{-s,0}(F)=0$ for $s>0$.
Similarly, 
by appeal to \eqref{triangle}, 
the group $\pi_{2,0}\vs_{0}\KQ $ is trivial.
Thus all the differentials entering or exiting the column $s=1$ are trivial. 
The splitting uses the determinant map, 
see, e.g., the proof of Lemma 3.8 in \cite{zbMATH05840069}. 
To calculate $\KQ_{2,1}(F)$ we proceed as above for $n=-1$ and $s=1$. 
When $n<-1$ and $s=1$ the vanishing is immediate since the column contains only trivial groups.

The second and third parts follow by similar arguments 
(see, e.g., p.~528 in \cite{zbMATH05840069} for the direct sum description of $\KQ_{2,0}(F)$).
\end{proof}

Next we turn to more complete calculations of 
hermitian $K$-groups of fields of cohomological 
dimension $\cd$ at most $2$;
i.e.,  
$H^{p,q}(F)=0$ for $p>2$.
Finite fields, local fields, and totally imaginary number fields are 
typical examples of such fields \cite[Section 14]{levine99}.
We claim that our assumption 
$\cd(F)\leq 2$ implies the vanishing 
$\widetilde{H}^{p,q}(F)=0$ for $p>2$:
By \cite[Theorem 17]{Bachmann}, 
it suffices to consider the diagonal,  
since when $p\neq q$ there is an 
isomorphism
$\widetilde{H}^{p,q}(F)\cong H^{p,q}(F)$.
By \cite[Chapter 5]{2020arXiv200406634B}, 
the diagonal of Milnor-Witt motivic cohomology 
$\widetilde{H}^{n,n}(F,\Z)$ 
is naturally isomorphic to Milnor-Witt 
$K$-theory $K_n^\MW(F)$. 
Moreover, 
by \cite[Cor.~2.3.7]{deglise2023notesmilnorwittktheory} (see \cite[Theorem 5.3]{Morel-Witt} if $\Char(F)\neq 2$), 
the fundamental ideal in the 
Grothendieck-Witt ring along with the 
Milnor and Milnor-Witt $K$-groups 
participate in a pullback square
\[\begin{tikzcd}
K_n^\MW(F)\ar{r}\ar{d} & K_n^\sfM(F)\ar{d}\\
I^n(F)\ar{r} & I^n(F)/I^{n+1}(F)
\end{tikzcd}\]
which gives rise to the short exact sequence
\begin{equation}\label{ses:I-MW-M}
0\to I^{n+1}(F)\to K_n^\MW(F)\to K_n^\sfM(F)\to0
\end{equation}
The sequence \eqref{ses:I-MW-M} is 
exact for all $n\in\Z$ under the 
convention that $I^n(F)=W(F)$ for $n< 0$
\cite[Chapter 2, Example 6.1.5]{2020arXiv200406634B}
(in which case $K_n^\sfM(F)=0$).
When $\cd(F)\le2$, 
\eqref{ses:I-MW-M} implies that 
$\widetilde{H}^{n,n}(F)=0$ for 
$n\geq 3$ since $I^3(F)=0$ by 
\cite[p. 81]{Milnor-Husemoller} and $K_n^\sfM(F)\cong H^{n,n}(F)=0$
in the same range 
by Suslin's identification of Milnor's $K$-groups with the diagonal 
of motivic cohomology, 
see \cite[Lecture 5]{zbMATH05051055}.


In the following computation of hermitian 
$K$-groups, 
we write $A\bullet B$ for an abelian 
group extension of $B$ by $A$, 
so that there exists a short exact 
sequence 
$$
0\to A\to A \bullet B\to B\to 0
$$

\begin{proposition}
\label{prop:KQn0}
If $F$ is a field with $\cd(F)\le2$, then the very effective slice spectral sequence for $\KQ$ collapses at its first page. 
For $n>0$, the hermitian $K$-groups are given up to extensions as follows.
\begin{center}
 \begin{tabularx}{\textwidth}{c | c | c | c | l l l l} 
 \toprule
 $n>0$ 
 & $\KQ_{n,0}$ & $\KQ_{n+2,1}$ & $\KQ_{n+4,2}$ & $\KQ_{n+6,3}$ \\
 \midrule
$8k$ & $h^{2,4k+1}\bullet h^{1,4k+1}$ & $H^{2,4k+1}\bullet h^{0,4k}$ & $0$  & $H^{2,4k+1}$ \\
$8k+1$ & $h^{1,4k+1}\bullet h^{0,4k+1}$ & $H^{1,4k+1}$ & $0$ & $h^{2,4k+2}\bullet {H}^{1,4k+1}$ \\
$8k+2$ & $H^{2,4k+2}\bullet h^{0,4k+1}$ & $0$ & $H^{2,4k+2}$ & $h^{2,4k+2}\bullet h^{1,4k+2}$\\
$8k+3$ & $H^{1,4k+2}$ & $0$ & $h^{2,4k+3}\bullet {H}^{1,4k+2}$ & $h^{1,4k+2}\bullet h^{0,4k+2}$ \\
$8k+4$ & $0$ & $H^{2,4k+3}$ & $h^{2,4k+3}\bullet h^{1,4k+3}$ & $H^{2,4k+3}\bullet h^{0,4k+2}$ \\
$8k+5$ & $0$ & $h^{2,4k+4}\bullet {H}^{1,4k+3}$ & $h^{1,4k+3}\bullet h^{0,4k+3}$ & $H^{1,4k+3}$\\
$8k+6$ & ${H}^{2,4k+4}$ & $h^{2,4k+4}\bullet h^{1,4k+4}$& $H^{2,4k+4}\bullet h^{0,4k+3}$ & $0$ \\
$8k+7$ & $h^{2,4k+5}\bullet {H}^{1,4k+4}$ & $h^{1,4k+4}\bullet h^{0,4k+4}$ & $H^{1,4k+4}$ & $0$ \\
\bottomrule
\end{tabularx}
\end{center}
\end{proposition}
\begin{proof}
All the differentials in the very effective slice spectral sequence for $\KQ$ are trivial owing to the assumption $\cd(F)\leq 2$. 
The calculation follows by a direct inspection of the $\tilde{E}^1$-terms in \Cref{fig:e1-KQ}, together with the convergence of the very effective slice spectral sequence \Cref{proposition:strongconvergence}. 
The calculation might involve determining the corresponding homotopy group of $\vs_0\KQ$, for which \Cref{fig:einfty-v} is helpful.
For example, when $n=8k+6$ and the weight $w=0$, 
\eqref{triangle} gives an isomorphism 
$$
\KQ_{8k+6,0}
\cong 
\pi_{8k+6,0}\vs_{4k+4}\KQ 
\cong 
{H}^{2,4k+4}
$$
on account of the vanishing $h^{3,4k+4}=0$, and, moreover, when $n=8k+7$ and $w=0$, an exact sequence
\[
0\to h^{2,4k+5}\to
\KQ_{8k+7,0}
\cong 
\pi_{8k+7,0}\vs_{4k+4}\KQ 
\to 
{H}^{1,4k+4}
\to 0
\]
This shows the cases $n=8k+6, 8k+7$ in the table for $\KQ_{n,0}$. 
The remaining groups are calculated similarly.
\end{proof}

\begin{example}\label{ex:kq-finite-fields}
If $F=\mathbb{F}_q$ is a finite field, hence of cohomological dimension $1$, the above computations recover Friedlander's computation of hermitian $K$-theory of finite fields in \cite{Friedlander-Fq}.
This reference addresses only the even weight part. 
Indeed, for $n=0$ we obtain $\KQ_{0,0}\cong\GW$, $\KQ_{2,1}\iso \Z/2$, $\KQ_{4,2}\cong \Z$ and $\KQ_{6,3}\iso 0$, while for $n>0$ and odd characteristic we have
\begin{center}
\begin{tabularx}{\textwidth}{l | c c c c c c c c}
\toprule
$n\mod 8$ & $0$ & $1$ & $2$ & $3$ & $4$ & $5$ & $6$ & $7$\\
\midrule
$\KQ_{n,0}$ & $\Z/2$ & $(\Z/2)^2$ & $\Z/2$ &  $\Z/(q^{\frac{n+1}{2}-1})$ & 0 & 0 & 0 & $\Z/(q^{\frac{n+1}{2}-1})$\\
$\KQ_{n+2,1}$ & $\Z/2$ & $\Z/(q^{\frac{n+1}{2}-1})$ & 0 & 0 & 0 & $\Z/(q^{\frac{n+1}{2}-1})$ & $\Z/2$ & $(\Z/2)^2$\\
$\KQ_{n+4,2}$ & $0$ & $0$ & $0$ &  $\Z/(q^{\frac{n+1}{2}-1})$ & $\Z/2$ & $(\Z/2)^2$ & $\Z/2$ & $\Z/(q^{\frac{n+1}{2}-1})$\\
$\KQ_{n+6,3}$ & $0$ &  $\Z/(q^{\frac{n+1}{2}-1})$ & $\Z/2$ & $(\Z/2)^2$ & $\Z/2$ & $\Z/(q^{\frac{n+1}{2}-1})$& $0$ & $0$ \\
\bottomrule
\end{tabularx}
\end{center}
The single potentially nontrivial extension can be shown to split with the aid of topology. More precisely, in odd characteristic the first stable homotopy group of the topological sphere spectrum appears as a direct summand via étale realization. The corresponding table for finite fields of characteristic $2$ is as follows.
\begin{center}
\begin{tabularx}{\textwidth}{l| c c c c c c c c}
\toprule
$n\mod8$ & $0$ & $1$ & $2$ & $3$ & $4$ & $5$ & $6$ & $7$\\
\midrule
$\KQ_{n,0}\cong \KQ_{n+4,2}$ & $0$ & $0$ & $0$ &  $\Z/(q^{\frac{n+1}2}-1)$ & 0 & 0 & 0 & $\Z/(q^{\frac{n+1}2}-1)$\\
$\KQ_{n+2,1}\cong \KQ_{n+6,3}$ & $0$ &  $\Z/(q^{\frac{n+1}2}-1)$ & 0 & 0 & 0 & $\Z/(q^{\frac{n+1}2}-1)$ & $0$ & $0$ \\
\bottomrule
\end{tabularx}
\end{center}
In characteristic $2$, étale realization cannot detect the first stable homotopy group of the topological sphere spectrum, since this requires inverting the characteristic. 
In particular, the vanishing $\KQ_{1,0}(\FF_{2})=0$ implies that the first topological Hopf map maps to zero in $\KQ_{\FF_{2}}$. 
This phenomenon already occurs for the motivic sphere spectrum $\unit_{\FF_{2}}$, as will be discussed in separate work.
\end{example}

\section{The very effective zero slice of hermitian $K$-theory}
\label{section:tvezsokq}
 
An understanding of the coefficients of $\vs_{0}\KQ$ is useful for the computation of the very effective slice spectral sequence. For a regular scheme $S$, one has $H^{2,1}(S)\cong Pic(S)$, the Picard group of isomorphism classes of line bundles on $S$, while $h^{2,1}(S)$ is its quotient modulo tensor squares.

\begin{lemma}
\label{lem:pi0-v}
Over a Dedekind scheme $\mathcal{D}$ with $h^{n+1,n}(\mathcal{D})=0$ for all $n\in \Z$, 
 the natural map  $\vs_{0}\KQ \to \cHZ$ induces an isomorphism on $\pi_{0+\bideg}$.
\end{lemma}
\begin{proof}
The assumption implies $\pi_{0+\bideg}\Sigma \MZ/2=0$. 
The statement then follows from the triangle 
$\Sigma \MZ/2\to \vs_{0}\KQ \to \cHZ$ in~(\ref{triangle}).
\end{proof}

\begin{lemma}
\label{lem:etainv-v}
There are equivalences of motivic spectra 
\[
\HWZ[\eta^{-1}] \simeq \cHZ[\eta^{-1}]\simeq\vs_{0}\KQ [\eta^{-1}]
\] 
Over a field $F$, they are equivalent to the motivic Eilenberg-MacLane spectrum associated with the homotopy module $\mathsf{W}(F)[\eta,\eta^{-1}]$ concentrated in degree zero.
\end{lemma}

\begin{proof}
The first statement follows from the cofiber sequences~\eqref{triangle2} and~\eqref{triangle}.
The second statement holds for fields of characteristic not 2 by \cite[Theorem 17]{Bachmann}.
For a field of any characteristic, the identifications 
\[ \pi_0\bigl(\vs_0\KQ[\eta^{-1}]\bigr)\cong \pi_0\bigl(\kq[\eta^{-1}]\bigr)\cong \pi_0\KW \cong \mathsf{W}(F)[\eta,\eta^{-1}]\] 
hold basically by construction.
Lower homotopy modules are zero by Morel's connectivity theorem \cite{morel.connectivity}.
The second statement follows once we show that for every element $a\in \pi_{s+(w)}\HWZ$ with $s>0$, there exists $\ell\in \mathbb{N}$ such that $\eta^\ell \cdot a=0$. If $s>0$, then for all $w\ge 0$,
the group $\pi_{s+(w)}\HWZ\cong \pi_{s+(w+1)}\f_1\vs_0\KQ\cong \pi_{s+(w+1)}\vs_0\KQ$ is zero.
In fact, the canonical map $\vf_1\KQ\to \vf_0\KQ$ induces an isomorphism $\pi_{s+(w+1)}\vf_1\KQ\to \pi_{s+(w+1)}\vf_0\KQ$ for $s>0$ and $w\geq 0$, and $\pi_{0+(w+1)}\vf_1\KQ=0$ by construction. 
Hence for $a\in \pi_{s+(w)}\HWZ$ with $w<0$, the element $\eta^{-w}\cdot a \in \pi_{s+(0)}\HWZ=0$.
\end{proof}

\begin{lemma}
\label{lem:convergence-v}
The slice completions of $\HWZ$, $\cHZ$, and $\vs_{0}\KQ $ 
are naturally equivalent to their respective $\eta$-completions.
\end{lemma}

\begin{proof}
For the $\eta$-completion of $\E$ defined in \cite[§2]{zbMATH06957413} there is a natural distinguished triangle
\begin{equation}
\label{equation:second}
\holim \Sigma^{(q)}\E  \to \E\to \E^\wedge_\eta 
\end{equation}
The slice completion of $\E$ defined in \cite[(3.10)]{zbMATH06957413} 
participates in a natural distinguished triangle
\begin{equation}
\label{equation:first}
\holim \f_{q}\E  \to \E\to \sc\E  
\end{equation}
If $\E$ is effective, there exists a natural transformation from \eqref{equation:second} to \eqref{equation:first}, which is the identity on $\E$. 
It is an equivalence for $\HWZ$ owing to the last statement in \Cref{prop:slices-cHZ}.
The triangles \eqref{triangle2} and \eqref{triangle} conclude the proof, using that $\MZ$ and $\MZ/2$ are both $\eta$- and slice-complete.
\end{proof}

Now we identify the first slice differential 
$\s_\ast(\vs_{0}\KQ )\to \Sigma \s_{\ast+1}\vs_{0}\KQ $.

\begin{lemma}
\label{lem:d1-v}
The first slice differential for $\vs_{0}\KQ $ is given by $(\Sq^2\pr^\infty_2,0)$ in degree zero, 
and
\[\begin{pmatrix} 
\Sq^2+\rho\Sq^1 & \Sq^3\Sq^1\\ \tau & \Sq^2
\end{pmatrix} \] in every positive degree.
\end{lemma}

\begin{proof}
It suffices to verify the claim over $\Spec(\Z)$, since the slices base change by Proposition~\ref{prop:slices-cHZ}, together with Spitzweck’s construction of the absolute Eilenberg–MacLane spectrum. The group of possible differentials over $\Spec(\Z)$ injects into the corresponding group over $\Spec(\Q)$ by \Cref{lem:endo-mz-weight1}. The desired statement then follows from the case of $\Spec(\Q)$, where \cite[Lemma 4.1]{zbMATH07003144} applies.
\end{proof}

\begin{lemma}
\label{lemma:sssforV}
Let $S$ be the spectrum of a field or the ring of integers.
The slice spectral sequence for $\vs_{0}\KQ $ collapses on its $E^{2}$-page.
\end{lemma}
\begin{proof}
Lemma~\ref{lem:d1-v} shows that the $E^2$-page of the $-n$th slice spectral sequence for $\vs_{0}\KQ$ has the form depicted in Figure~\ref{fig:einfty-v}, with the pattern extending naturally along the zeroth column and the horizontal axis. Except for the columns indexed by $0$ and $1$, the $E^2$-page is concentrated in two rows. In the column indexed by $1$, the group $h^{n+q-1,n+q}/\tau h^{n+q-1,n+q-1}$ 
is nontrivial if and only if the base is the ring of integers and $n+q=2$, in which case it has order two. No higher differential can hit this group, as all potential sources vanish for degree reasons. Lemma~\ref{lem:pi0-v} implies that all higher differentials entering the leftmost column vanish over a field; the same conclusion then holds over the integers by base change to $\Spec(\Q)$. It follows that all higher differentials vanish, so that $E^2=E^\infty$ for $\vs_{0}\KQ$.
\end{proof}

\begin{figure}
\begin{center}
\resizebox{\linewidth}{!}{\begin{tikzpicture}[font=\scriptsize,scale=1.2]
        \draw[help lines,xstep=2.0,ystep=1.0] (-0.5,-0.0) grid (8.5,4.5);
        \foreach \i in {0,...,4} {\node[label=left:$\i$] at (-0.5,1.0*\i+.5) {};}
        \foreach \i in {0,...,4} {\node[label=below:$\i$] at (2.0*\i+.8,-.2) {};}
        
{\node[above right=0pt] at (0.0,0.0) {${H}^{n,n}$};}
{\node[above right=0pt] at (0.0,1.0) {${h}^{n+1,n+1}$};}
{\node[above right=0pt] at (0.0,2.0) {${h}^{n+2,n+2}$};}
{\node[above right=0pt] at (0.0,3.0) {${h}^{n+3,n+3}$};}
{\node[above right=0pt] at (0.0,4.0) {${h}^{n+4,n+4}$};}

{\node[above right=0pt] at (2.0,0.0) {$H^{n-1,n}$};}
{\node[above right=0pt] at (2.0,1.0) {${h}^{n,n+1}\!/\!\Sq^2\!\pr^\infty_2$};}
{\node[above right=0pt] at (2.0,2.0) {$h^{n+1,n+2}\!/\!\tau$};}
{\node[above right=0pt] at (2.0,3.0) {$h^{n+2,n+3}\!/\!\tau$};}
{\node[above right=0pt] at (2.0,4.0) {$h^{n+3,n+4}\!/\!\tau$};}

{\node[above right=0pt] at (4.0,0.0) {$\ker(\Sq^2\!\pr^\infty_2)$};}
{\node[above right=0pt] at (4.0,1.0) {${h}^{n-1,n+1}$};}
{\node[above right=0pt] at (4.0,2.0) {$0$};}
{\node[above right=0pt] at (4.0,3.0) {$0$};}
{\node[above right=0pt] at (4.0,4.0) {$0$};}

{\node[above right=0pt] at (6.0,0.0) {$H^{n-3,n}$};}
{\node[above right=0pt] at (6.0,1.0) {${h}^{n-2,n+1}$};}
{\node[above right=0pt] at (6.0,2.0) {$0$};}
{\node[above right=0pt] at (6.0,3.0) {$0$};}
{\node[above right=0pt] at (6.0,4.0) {$0$};}

{\node[above right=0pt] at (8.0,0.0) {$H^{n-4,n}$};}
{\node[above right=0pt] at (8.0,1.0) {$h^{n-3,n+1}$};}
{\node[above right=0pt] at (8.0,2.0) {$0$};}
{\node[above right=0pt] at (8.0,3.0) {$0$};}
{\node[above right=0pt] at (8.0,4.0) {$0$};}

\end{tikzpicture}}
\end{center}
\caption{$E^\infty_{p-n,q,-n}(\vs_{0}\KQ )$}
\label{fig:einfty-v}
\end{figure}

Modules over Milnor-Witt $K$-theory are used in the calculations of 
universal motivic invariants in \cite{zbMATH07003144,zbMATH07858209}.
We note the following result for $\vs_{0}\KQ $.

\begin{proposition}\label{thm:pi-v}
Let $F$ be a field.
If $s>0$ the $\KMW$-module $\pi_{s-\bideg}\vs_{0}\KQ $ participates in the exact sequences
\begin{align*}
& 0 \to \mathbf{h}^{\star-s+1,\star+1}/\Sq^2\pr^\infty_2 \to 
\pi_{s-\bideg}\vs_{0}\KQ  \to \mathbf{H}^{\star-s,\star} \to 0 \\
& 0 \to \mathbf{h}^{\star-s+1,\star+1} \to 
\pi_{s-\bideg}\vs_{0}\KQ  \to 
\ker(\Sq^2\pr^\infty_2:\mathbf{H}^{\star-s,\star}\to \mathbf{h}^{\star-s+2,\star+1}) \to 0 \\
& 0 \to \mathbf{h}^{\star-s+1,\star+1} \to 
\pi_{s-\bideg}\vs_{0}\KQ  \to \mathbf{H}^{\star-s,\star} \to 0 
\end{align*}
for $s \equiv 1 \bmod 4$, $s \equiv 2 \bmod 4$, and $s \equiv 0,3 \bmod 4$, 
respectively. 
The Hopf map $\eta$ operates through the projection from the quotient in the short 
exact sequence in weight $n$ to the kernel in the short exact sequence in weight $n-1$.
\end{proposition} 

\begin{proof}
By Lemma~\ref{lem:convergence-v} the slice spectral sequence of $\vs_{0}\KQ $ computes the homotopy of $\vs_{0}\KQ ^\wedge_\eta$. 
\Cref{lemma:sssforV} and its proof shows the only column on the $E^\infty$-page with potentially infinitely many terms is located in degree zero.
It follows, see Lemma~\ref{lem:pi0-v}, that $\vs_{0}\KQ ^\wedge_\eta[\eta^{-1}]$ is the motivic Eilenberg-MacLane spectrum associated to the homotopy module $\mathsf{W}^\wedge_{\mathsf{I}}[\eta,\eta^{-1}]$ associated to the completion of the Witt ring at its fundamental ideal.
Hence all the positive homotopy modules of $\vs_{0}\KQ ^\wedge_\eta[\eta^{-1}]$ are zero.
Since $\vs_{0}\KQ[\eta^{-1}]$ is also a motivic Eilenberg–MacLane spectrum concentrated in degree zero (Lemma~\ref{lem:etainv-v}), the positive homotopy modules of $\vs_{0}\KQ$ and $\vs_{0}\KQ^\wedge_\eta$ coincide. The short exact sequences follow directly from the $E^\infty$-page, and \Cref{prop:slices-cHZ} accounts for the claims regarding $\eta$.
\end{proof}

\begin{remark}
\label{rem:pi-v}
In general, 
the short exact sequences in Theorem~\ref{thm:pi-v} do not split, 
not even as sequences of abelian groups. 
For example, 
if $F$ is a number field, 
the canonical map $\kq\to \vs_{0}\KQ $ induces an isomorphism $\pi_{1-(2)}\kq \to \pi_{1-(2)}\vs_{0}\KQ $, 
and the forgetful map $\pi_{1-(2)}\kq\to \pi_{1-(2)}\kgl\iso K_3$ participates in the commutative diagram
\[
\begin{tikzcd}
0\ar[r] & h^{2,3} \ar[r] \ar[d,"\partial^2_\infty"] & \pi_{1-(2)}\kq \ar[r]\ar[d] & H^{1,2} 
\ar[r] \ar[d,"="] & 0 \\
0\ar[r] & H^{3,3} \ar[r] & \pi_{1-(2)}\kgl \ar[r] & H^{1,2} \ar[r] & 0
\end{tikzcd}\]
If $F$ admits a real embedding, 
the connecting map on the left-hand side is surjective and does not split the bottom sequence by \cite{lee-szczarba}; 
hence the top sequence cannot split. 
\end{remark}

\begin{proposition}
\label{prop:veff-forg}
Let $\Forg:\KQ\to\KGL$ denote the forgetful map.
\begin{enumerate}
\item The induced map $\vs_0(\Forg)$ fits into the commutative diagram
\[\begin{tikzcd}
\vs_0\KQ \ar{r}{\vs_0(\Forg)}\ar{d} & \vs_0\KGL\ar{d}{\simeq}\\
\cHZ\ar{r} & \HZ
\end{tikzcd}\]
The lower horizontal map is the canonical map from Milnor-Witt motivic cohomology to motivic cohomology.
\item The map $\vs_1(\Forg)\colon\vs_1\KQ \simeq\Sigma^{1+(1)}\HZ/2\to\vs_1\KGL\simeq\Sigma^{1+(1}\HZ$ 
is trivial.
\item The map $\vs_2(\Forg)\colon \vs_2\KQ \simeq\Sigma^{2+(2)}\HZ\to \vs_2\KGL\simeq\Sigma^{2+(2)}\HZ$ 
is multiplication by $2$.
\end{enumerate}
\end{proposition}

\begin{proof}
For the zeroth slice, 
applying $[-,\HZ]$ to the cofiber sequence \eqref{triangle} yields the long exact sequence
\[
[\Sigma^{2}\HZ/2,\HZ]\to[\cHZ,\HZ]\to[\vs_0\KQ ,\HZ]\to[\Sigma\HZ/2,\HZ].
\]
As $[\HZ/2,\Sigma^{t}\HZ]=0$ for $t\ne 1$, 
there is a canonically induced isomorphism
$$[\cHZ,\HZ]\xrightarrow{\cong} [\vs_0\KQ ,\HZ]$$ 
Moreover, since $\HZ\simeq\vs_0\KGL\simeq\s_0\KGL$ is a $0$-slice, the canonical map to 
$\s_0\cHZ$ in the cofiber sequence
\[ \f_1\cHZ\to \cHZ \to \s_0\cHZ \]
induces an isomorphism 
$$[\s_0\cHZ,\HZ]\xrightarrow{\cong} [\cHZ,\HZ]$$
Combined with Proposition~\ref{prop:slices-cHZ} we deduce the isomorphisms
\begin{equation}
\label{equation:ringmapinteger}
[\s_0\cHZ,\HZ]\cong [\HZ\oplus \Sigma^{2}\HZ/2,\HZ]\cong [\HZ,\HZ] \cong \Z
\end{equation}
The forgetful map $\Forg:\KQ \to \KGL$ is a ring spectrum homomorphism, 
and hence also the induced
map on the (very) effective zero-slices by \cite{grso}.\footnote{It suffices to use that $\Forg$ is unital.} 
The first part follows from \eqref{equation:ringmapinteger}.

The second part follows since $[\HZ/2,\HZ]=0$. 
For the third part, 
by an inspection of the proof of \cite[Theorem 16]{Bachmann}, 
based on \cite[Theorem 4.4]{RO:slices},  
there is a cofiber sequence
\[
\vs_0(\Sigma^{2+(2)}\KQ) \xrightarrow{\vs_0(\Sigma^{2+(2)}\Forg)} \vs_0(\Sigma^{2+(2)}\KGL) 
\xrightarrow{\vs_0(\Sigma^{3+(3)}\Hyper \circ \beta)} \vs_0(\Sigma^{3+(3)}\KQ) 
\]
(applying $\vs_0\Sigma^{2+(2)}$ to the Wood cofiber sequence), and hence a cofiber sequence:
\[
\HZ\xrightarrow{\Sigma^{2+(2)}\vs_{-2}(\Forg)} \HZ \to \HZ/2 
\]
It follows that $\vs_2(\Forg)$ is multiplication by $\pm 2$. 
Since the unit map in symplectic $K$-theory induces compatible identifications on slices, 
the map $\vs_2(\Forg)$ is multiplication by $2$.
\end{proof}

\begin{proposition}\label{prop:veff-hyper}
Let $\Hyper:\KGL \to \KQ$ denote the hyperbolic map.
\begin{enumerate}
\item\label{it:vs1-hyper} 
$\vs_1(\Hyper)\colon\vs_1\KGL\simeq\Sigma^{1+(1)}\HZ\to\vs_1\KQ \simeq\Sigma^{1+(1)}\HZ/2$ 
is the canonical projection.
\item\label{it:vs2-hyper} 
$\vs_2(\Hyper)\colon \vs_2\KGL\simeq\Sigma^{2+(2)}\HZ\to \vs_2\KQ \simeq\Sigma^{2+(2)}\HZ$ is the identity.
\end{enumerate}
\end{proposition}

\begin{proof}
The map $\vs_1(\Hyper)$ is either zero or the canonical projection. 
The long exact sequence of homotopy groups obtained from Theorem~\ref{item:wood-seq} in weight $2$ ends with 
\[ 
\dotsm \rightarrow
\pi_{2+(2)}\KQ 
\rightarrow
\pi_{2+(2)} \KGL 
\rightarrow
\pi_{1+(1)} \KQ 
\rightarrow
\pi_{1+(2)}\KQ 
\]
By calculations of symplectic $K$-groups or using the very effective slice spectral sequence 
and Proposition~\ref{prop:veff-forg}, 
this sequence takes the form 
\begin{equation}
\label{equation:symplecticles}
\dotsm 
\rightarrow
\Z 
\rightarrow
\Z 
\rightarrow
\Z/2 
\rightarrow
0
\end{equation}
As the very effective slice spectral sequence implies,
the surjection onto $\Z/2$ in \eqref{equation:symplecticles} 
is induced by $\pi_{1+(1)}\vs_1(\Hyper)$, 
which is thus nonzero.
Alternatively, the last cofiber sequence in the proof of \Cref{prop:veff-forg} implies the first claim.
Proposition~\ref{prop:veff-forg} and
the fact that $\s_2(\Forg \circ \Hyper)=\vs_2(\Forg \circ \Hyper)$
is multiplication by $2$ according to \cite[Proposition 4.4]{RO:slices} imply the second claim.
\end{proof}

\section{The first very effective slice differential of hermitian $K$-theory}
\label{section:tfvesdfkq}
In this section, 
we identify the first differential in the very effective slice spectral sequence of $\KQ$ as maps of motivic spectra.
The form and the $4+(4)$-periodicity of the very effective slices of $\KQ$ implies that it suffices to determine 
\[\vd^1_0\colon \vs_0\KQ \to \Sigma \vs_1\KQ  \simeq \Sigma^{3,1} \MZ/2\] and 
$$\vd^1_1\colon \vs_1\KQ \simeq\Sigma^{(1)}\MZ/2 \to \Sigma \vs_2\KQ \simeq\Sigma^{3+(2)}\MZ$$
We begin by calculating maps between Milnor-Witt and ordinary motivic cohomology spectra.
For our purposes, 
we extend certain computations known over a field or under conditions on the characteristics 
of residue fields in the base scheme \cite[Appendix A]{RO:slices}, \cite{Voevodsky:power}.

\begin{lemma}\label{lem:endo-mz-weightneg}
Let $S$ be a Dedekind scheme, and let $A$ and $B$ be abelian groups.
Then $[\MA,\Sigma^{s+(w)}\MB]_S=0$ for all $s\in \Z$ and all negative $w\in \Z$.
\end{lemma}

\begin{lemma}\label{lem:endo-mz-weight0}
Let $S$ be a connected Dedekind scheme, and let $A$ and $B$ be abelian groups.
There is a natural isomorphism
\[
[\MA,\Sigma^{s}\MB]_S \cong \begin{cases}
\Hom_{\Ab}(A,B) & s = 0 \\
\Ext_{\Ab}(A,B) & s = 1\\
0 & t \neq 0, 1
\end{cases}
\]
\end{lemma}
\begin{proof}
Choose a free resolution
\[ 0 \to F_1\to F_0 \to A\to 0\]
It induces a change-of-coefficients cofiber sequence of motivic Eilenberg-MacLane spectra.
One may conclude from the associated long exact sequence induced on $[-,\MB]_S$, as soon as $[\MZ,\Sigma^s\MB]_S$ has the correct values.
Since the unit morphism $\s_0\unit\to \s_0\MZ\simeq \MZ$ is an equivalence over any Dedekind domain by \cite[Theorem B.5]{bachmann-hoyois}, and $[\f_1\unit,\Sigma^s\MB]=0$ for all $s\in \Z$, there results a natural identification
\[ [\MZ,\Sigma^s\MB]_S\cong [\unit,\Sigma^s\MB]_S\cong \begin{cases} B & s=0 \\ 0 & s\neq 0 \end{cases}\]
\end{proof}

Although the following lemmata apply more generally to Dedekind domains, the restriction to $\Spec(\mathbb{Z})$ avoids complications from a nontrivial Picard group.

\begin{lemma}\label{lem:endo-mz-weight1}
In $\SH(\Z)$ there are isomorphisms
\begin{align*}
[\MZ,\Sigma^{s+1,1}\MZ] &\iso  \begin{cases} H^{1,1}(\Z) \iso \Z/2 & s=0 \\
\Z/2 & s=2 \\
0 & s\notin \{0,2\} \end{cases}
\\
[\MZ,\Sigma^{s+1,1}\MZ/2] &\iso  \begin{cases} h^{s+1,1}(\Z) \iso \Z/2 & s\in \{-1,0\} \\
\Z/2 & s\in \{1,2\} \\
0 & s\notin \{-1,0,1,2\} \end{cases}
\\
[\MZ/2,\Sigma^{s+1,1}\MZ] &\iso \begin{cases} \Z/2 & s\in \{0,1,2,3\} \\ 
0 & s\notin \{0,1,2,3\} \end{cases}
\end{align*}
\end{lemma}

\begin{proof}
For the first statement, the cofiber sequence 
\[ \f_2\unit \to \f_1\unit \to \s_1\unit \to \Sigma\f_2\unit\]
induces a long exact sequence 
\[ \dotsm \to [\f_2\unit,\Sigma^{s,1}\MZ] \to [\s_1\unit,\Sigma^{s+1,1}\MZ] \to [\f_1\unit,\Sigma^{s+1,1}\MZ] \to [\f_2\unit,\Sigma^{s+1,1}\MZ] \to \dotsm \]
in which the outer terms displayed are zero, since any $2$-effective motivic spectrum 
maps trivially to any $1$-slice.
Hence we have $[\f_1\unit,\Sigma^{s+1,1}\MZ]\iso [\s_1\unit,\Sigma^{s+1,1}\MZ]$.
\Cref{lem:unit-kq-slices} shows $\s_1\unit\simeq \Sigma^{1,1}\MZ/2$, 
whence $[\f_1\unit,\Sigma^{s+1,1}\MZ]$ is given by \Cref{lem:endo-mz-weight0}.
In particular, the group is nontrivial only for $s=1$, in which case it is of order two.
The cofiber sequence 
\[ \f_1\unit \to \unit \to \s_0\unit\simeq \MZ \to \Sigma\f_1\unit\]
and the isomorphism 
\[ [\unit,\Sigma^{s+1,1}\MZ]\iso H^{s+1,1}(\Z)\iso \begin{cases} \Z^{\times} & s=0\\ 0 & s\neq 0\end{cases} \]
then induce the exact sequence
\[ 0\to [\MZ,\Sigma^{1,1}\MZ] \to H^{1,1}(\Z) \to 0\to [\MZ,\Sigma^{2,1)}\MZ] \to 0 
\to \Ext(\Z/2,\Z) \to [\MZ,\Sigma^{3,1}\MZ] \to 0 \]
giving the first identification.
For $s=0$, the isomorphism is induced by multiplication with the unit $-1$, 
while the second nontrivial operation factors through
$\Sigma^{1,1}(\partial^2_\infty \MZ/2\to \Sigma\MZ)$.
The remaining two isomorphisms follow from the cofiber sequence
\[ \MZ\xrightarrow{2}\MZ \xrightarrow{\pr^\infty_2}{\MZ/2}\xrightarrow{\partial^2_\infty}\Sigma \MZ.\]
In particular, the unique nonzero operation $\MZ\to \Sigma^{3,1}\MZ$ factors as
\[ \MZ \xrightarrow{\pr^\infty_2}\MZ/2 \xrightarrow{\Sq^2} \Sigma^{2,1}\MZ/2 \xrightarrow{\partial^2_\infty} \Sigma^{3,1}\MZ\]
where the middle term is the unique nonzero map.
\end{proof}
\begin{lemma}\label{lem:unique-v-operation}
In $\SH(\Z)$ there is a unique nonzero map
\[ \vs_0\KQ\to \Sigma\vs_1\KQ \simeq \Sigma^{3,1}\MZ/2\]
\end{lemma}

\begin{proof}
The vanishing $[\E,\vs_{1}\KQ]=0$ for $2$-very effective $\E$ \cite[Lemma 7]{Bachmann} applied 
to the long exact sequences arising from the homotopy cofiber sequences
\[ \vf_{n+1}\KQ \to \vf_n\KQ \to \vs_n\KQ \]
together with the contractibility of $\holim_n \vf_n\KQ$, yield isomorphisms
$\Z/2\iso [\vs_1\KQ,\vs_1\KQ] \iso [\vf_1\KQ,\vs_1\KQ]$.
The computation $[\kq,\Sigma\vs_1]=0$ can be derived from \Cref{thm:woodkq} and similar statements for $\kgl$, which in turn rely on the canonical map 
\[ \MZ\simeq \s_0\kgl \to \Sigma^{3,1}\kgl \simeq \Sigma\f_1\kgl \to \Sigma\s_1\kgl\simeq \Sigma^{3,1}\MZ\]
being the unique nonzero element from \Cref{lem:endo-mz-weight1}.
It follows that a group of order two surjects onto $[\vs_{0}\KQ,\Sigma\vs_{1}\KQ]$.
Since over fields of characteristic not two this surjection is an isomorphism of groups with 
precisely two elements, the result follows.
\end{proof}

For the identification of a second differential, the following will be used.

\begin{lemma}\label{lem:endo-weight2-mz}
The group $[\MZ,\Sigma^{5,2}\MZ]$ in $\SH(\Z)$ is cyclic of order six.
A generator is given by the sum of the elements
\[ \MZ\xrightarrow{\pr^\infty_2}\MZ/2\xrightarrow{\Sq^4}\Sigma^{4,2}\MZ/2\xrightarrow{\partial^2_\infty} \Sigma^{5,2}\MZ\]
and 
\[ \MZ\xrightarrow{\pr^\infty_3}\MZ/3
\xrightarrow{{\mathsf{P}}^1}\Sigma^{4,2}\MZ/3\xrightarrow{\partial^3_\infty}\Sigma^{5,2}\MZ\]
\end{lemma}

\begin{proof}
We will use the identifications $\s_0\unit\simeq \MZ$ and $\s_1\unit\simeq \Sigma^{1,1}\MZ/2$ 
from \Cref{lem:unit-kq-slices} and 
$\s_2\unit\simeq \Sigma^{2,2}\oplus \Sigma^{3,2}\MZ/12$ from \Cref{lem:2-slice-sphere}, 
together with their defining cofiber sequences.
Since $3$-effective motivic spectra map trivially to $2$-slices, 
we have 
$$
[\s_2\unit,\Sigma^{s+2,2}\MZ]
\cong 
[\f_2\unit,\Sigma^{s+2,2}\MZ]
$$ 
for any $s\in \Z$.
Then the vanishing $[\unit,\Sigma^{s+2,2}\MZ]=0$ for $s\notin \{-1,0\}$ (actually, just $s>0$) 
is used to conclude $[\MZ,\Sigma^{s+2,2}\MZ]\iso [\f_1\unit,\Sigma^{s+1,2}\MZ]$ and an exact sequence
\[ [\s_2\unit,\Sigma^{s,2}\MZ] \to [\s_1\unit,\Sigma^{s+1,2}\MZ] \to 
[\MZ,\Sigma^{s+2,2}\MZ] \to [\s_2\unit,\Sigma^{s+1,2}\MZ] \to [\s_1\unit,\Sigma^{s+2,2}\MZ] \]
for $s>1$.
\Cref{lem:endo-mz-weight1} and \Cref{lem:endo-mz-weight0} show that $[\MZ,\Sigma^{s+2,2}\MZ]=0$ for $s>3$, and identify the above exact sequence with
\begin{equation}\label{eq:endo-mz-weight2}
\Ext(\Z\! / 2,\Z) \to [\MZ/2,\Sigma^{3,1}\MZ] \to [\MZ,\Sigma^{5,2}\MZ]\to \Ext(\Z\! /\! 12,\Z)\to [\MZ/2,\Sigma^{4,1}\MZ] 
\end{equation}
for $s=3$, 
where the first two groups, as well as the last, are of order two.
The first and last maps are induced by the initial slice differential $\s_{1}\unit \to \Sigma\s_{2}\unit$, which, by \Cref{lem:endo-mz-weight1} and \cite{zbMATH07003144}, corresponds in $\SH(\Z)$ to the pair $(\Sq^{2},, \inc^{2}_{12}\Sq^{2}\Sq^{1})$.
Consequently, the first map in \eqref{eq:endo-mz-weight2} detects the nontrivial class $\partial^{2}_{\infty}\Sq^{2}$, while the last detects $\partial^{2}_{\infty}\Sq^{2}\Sq^{1}=\partial^{12}_{\infty}\inc^{2}_{12}\Sq^{2}\Sq^{1}$.
Hence the group is cyclic of order six.
Via the cofiber sequence
\[ \MZ\xrightarrow{p}\MZ\xrightarrow{\pr^\infty_p}\MZ/p \xrightarrow{\partial^p_\infty} \Sigma\MZ\]
for $p\in \{2,3\}$, this yields elements
$\Sq^4\colon \MZ/2\to \Sigma^{4,2}\MZ/2$ and $\mathsf{P}^1\colon \MZ/3\to \Sigma^{4,2}\MZ/3$ 
which pull back to the classical motivic Steenrod operations of the same name, giving rise to the displayed nontrivial elements of orders $2$ and $3$, respectively.
\end{proof}

Now we extend the above computation to Milnor-Witt motivic cohomology.

\begin{lemma}\label{lem:mztilde-weight0}
Let $S$ be a connected Dedekind scheme. 
In $\SH(\S)$ there are isomorphisms
\[
[\cHZ, \Sigma^{t}\HZ/2] 
\cong
\begin{cases}
\Z/2 & t = 0 \\
\Z/2 & t = 2 \\
\Z/2 & t = 3 \\
0 & t \notin \{0, 2, 3\}
\end{cases}
\]

\[
[\cHZ, \Sigma^{t}\HZ] \cong
\begin{cases}
\Z & t = 0 \\
\Z/2 & t = 3 \\
0 & t \neq 0, 3
\end{cases}
\]
\end{lemma}

\begin{proof}
The slice spectral sequence
\begin{equation}
\label{eq:ss-cHZ}
E^{t,q,w}_1 \cong 
[\s_{q}(\cHZ),\Sigma^{t,w}\HZ/2]
\implies 
[\cHZ, \Sigma^{t,w}\HZ/2]
\end{equation}
with differentials $E^{t,q,w}_r \to E^{t+1,q-r,w}_r$ is strongly convergent.
This follows from Lemma~\ref{lem:convergence-v}, the fact that $\MZ/2$ is $\eta$-complete, and the vanishing $[\s_{q}(\cHZ),\Sigma^{s+1,1}\MZ/2]=0$ for $q>1$. 

The computation essentially follows from the determination of the slices of $\vs_0\KQ$ in \Cref{prop:slices-cHZ} together with the vanishing of motivic cohomology operations in negative weights (see \Cref{lem:endo-mz-weightneg}).
The group $E^{t,q,w}_1 = 0$ for $q < 0$.
For $q=0$, where $\s_0\cHZ\iso \s_0\vs_0\KQ\oplus \Sigma^{2,0}\MZ/2$, one obtains 
\[ E^{t,0,w}_1=[\MZ\oplus \Sigma^{2,0}\MZ/2,\Sigma^{t,w}\MZ/2] = [\MZ,\Sigma^{t,w}\MZ/2]\oplus [\MZ/2,\Sigma^{t-2,w}\MZ/2] \]

In weight $w=0$, 
the spectral sequence \eqref{eq:ss-cHZ} has $E^1$-page concentrated along the row $q = 0$, 
with terms given by 
\[
E^{0,0,0}_1\cong [\HZ, \HZ/2],\quad
E^{1,0,0}_1=0,\quad
E^{2,0,0}_1\cong [\HZ/2, \HZ/2],\quad
E^{3,0,0}_1\cong [\HZ/2, \Sigma\HZ/2]
\]
All remaining groups on the $E_1$-page vanish. It follows that the spectral sequence collapses, and the first claim then follows from \Cref{lem:endo-mz-weight0}. The second claim is established by an analogous computation.
\end{proof}

\begin{remark}
\label{rem:v-fiber}
In particular, over any connected Dedekind scheme there is a unique nontrivial map $\cHZ \to \Sigma^2 \HZ/2$, as described in \Cref{lem:mztilde-weight0}; this is precisely the connecting map $\partial$ arising from the homotopy cofiber sequence \eqref{triangle}. Moreover, the computation of $\pi_1 \vs_{0}\KQ$ from Theorem~\ref{thm:pi-v} together with \cite[Theorem 17]{Bachmann} allows one to identify the terms in the associated long exact sequence of homotopy modules
\[ 
\dotsm \to \pi_2\cHZ \cong \mathbf{H}^{\star-2,\star} 
\xrightarrow{\partial}
\pi_1\Sigma\MZ/2 \cong \mathbf{h}^{\star,\star} 
\to \pi_1\vs_{0}\KQ  \to \pi_1\cHZ  \cong \mathbf{H}^{\star-1,\star} \to 0
\]
Since $\pi_1\vs_{0}\KQ $ is typically a nontrivial extension of $\pi_1\cHZ$ by an often proper quotient of $\pi_1\Sigma\MZ/2 \cong h^{\star,\star+1}$, 
see Remark~\ref{rem:pi-v}, 
the map $\pi_1\Sigma\MZ/2\to \pi_1\vs_{0}\KQ $ is usually not injective.
More precisely, 
since there is a unique nonzero map $\cHZ\to \Sigma^{2,1}\HZ/2$, 
the connecting map fits into the commutative diagram
\[
\begin{tikzcd}
\cHZ \ar[r,"\partial"] \ar[d,swap,"\mathrm{can}"] & \Sigma^2 \HZ/2 \ar[d,"\tau"] \\
\HZ \ar[r,"\Sq^2\pr^\infty_2"] & \Sigma^{2,1}\HZ/2
\end{tikzcd}
\]
\end{remark}




\begin{lemma}
\label{lem:vtomod2weight1}
Let $S$ be a connected Dedekind scheme with $h^{2,1}(S)=0$.
We have the following identifications
\[ [\vs_{0}\KQ ,\Sigma^{s+1,1}\MZ/2] \cong \begin{cases}
[\s_1(\vs_{0}\KQ ),\Sigma^{4,1}\MZ/2] \cong \Z/2\{(0,\Sq^1)\} & s=3 \\
[\s_1(\vs_{0}\KQ ),\Sigma^{3,1}\MZ/2] \cong \Z/2\{(0,\id)\} & s=2 \\
[\s_0(\vs_{0}\KQ ),\Sigma^{1,1}\MZ/2] \cong h^{1,1}(S)\{\pr^\infty_2\} & s=0 \\
[\s_0(\vs_{0}\KQ ),\Sigma^{0,1}\MZ/2] \cong h^{0,1}(S)\{\pr^\infty_2\} & s=-1 \\
0 & \mathit{else}
\end{cases}
\]
\end{lemma}
\begin{proof}
As in the proof of \Cref{lem:mztilde-weight0}, consider the slice spectral sequence 
\[ 
[\s_{q}(\vs_{0}\KQ ),\Sigma^{s+1,1}\MZ/2] \Rightarrow [\vs_{0}\KQ ,\Sigma^{s+1,1}\MZ/2] 
\]
It converges strongly on account of Lemma~\ref{lem:convergence-v}, 
the fact that $\MZ/2$ is $\eta$-complete, 
and the vanishing $[\s_{q}(\vs_{0}\KQ ),\Sigma^{s+1,1}\MZ/2]=0$ for $q>1$. 
The latter implies it is concentrated in two rows, $q\in \{0,1\}$.
The differential is induced by the first slice differential for $\vs_{0}\KQ $ in degree zero, 
see Lemma~\ref{lem:d1-v}. 
It is an isomorphism for $s\in \{1,2\}$, which implies the result.
\end{proof}

\begin{corollary}\label{cor:vd-KQ}
Let $S$ be essentially smooth over a Dedekind scheme.
The first very effective slice differential
\[ \vd_0^1 \colon  \vs_0\KQ  \to \Sigma\vs_1\KQ \simeq\Sigma^{3,1}\MZ/2 \]
is the unique nontrivial map.
Moreover, 
the first very effective slice differential
\[ \vd_1^1 \colon \vs_1\KQ \simeq\Sigma^{2,1}\MZ/2 \to \Sigma\vs_2\KQ \simeq\Sigma^{5,2}\MZ\]
is the unique nontrivial map $\partial^2_\infty \Sq^2$.
\end{corollary}

\begin{proof}
Since the slice filtration is preserved under essentially smooth base change, it suffices to prove the statement over a Dedekind scheme.
The very effective slices of $\KQ$, including the first differential, are preserved under base change along morphisms of Dedekind schemes.
Hence it suffices to consider the base scheme $\Spec(\Z)$.
Lemma~\ref{lem:vtomod2weight1} implies there are only two choices for $\vd_0^1$.
Similarly, 
$\vd_1^1$ is an element in the group $[\MZ/2,\Sigma^{3,1}\MZ]\cong \Z/2\{\partial^\infty_2\Sq^2\}$ 
of order two by \Cref{lem:endo-mz-weight1}. 

To determine $\vd^1_1$ we consider the commutative diagram
\[
\begin{tikzcd}
\vs_1(\KGL/2)\ar[d,"\vd^1_1(\KGL/2)"'] & 
\vs_1\KGL \ar[rr,"\vs_1(\Hyper)"] \ar[d,"\vd^1_1\KGL"] \ar[l,"\pr^\infty_2"'] && 
\vs_1\KQ  \ar[d,"\vd^1_1\KQ "] \\
\Sigma\vs_2(\KGL/2) &
\Sigma\vs_2\KGL \ar[rr,"\Sigma\vs_2(\Hyper)"] \ar[l,"\pr^\infty_2"'] && 
\Sigma\vs_2\KQ 
\end{tikzcd}
\]
Proposition~\ref{prop:veff-hyper} implies that $\Sigma \vs_2(\Hyper)$ is an equivalence.
If $\vd^1_1 \KQ$ were zero, the same would hold for $\vd^1_1 \KGL$, contradicting the fact that $\vd^1_1(\KGL/2)$ is the first Milnor operation \cite[Lemma 5.1]{RO:slices}, whose composition with $\pr^\infty_2$ is the unique nontrivial map $\Sq^1 \Sq^2 \pr^\infty_2$; here we use that effective and very effective covers agree for $\KGL$ \cite[Lemma 4]{aro}.

To show that $\vd_0^1 \KQ$ is the unique nontrivial map, we consider the commutative diagram of very effective slices and differentials
\[
\begin{tikzcd}
\vs_0\KGL \ar[rr,"\vs_0(\Hyper)"] \ar[d,"\vd^1_0\KGL"']  && 
\vs_0\KQ  \ar[d,"\vd^1_0\KQ "] \\
\Sigma\vs_1\KGL \ar[rr,"\Sigma\vs_1(\Hyper)"] && 
\Sigma\vs^1_0\KQ 
\end{tikzcd}
\]
By Proposition~\ref{prop:veff-hyper}, 
the lower horizontal map $\Sigma\vs_1(\Hyper)$ is the unique nonzero map $\pr^\infty_2$ from 
$\Sigma\vs_1\KGL$ to $\Sigma\vs^1_0\KQ $.
The composite $\Sigma\vs_1(\Hyper)\circ\vd^1_0\KGL$ is thus the nonzero map $\Sq^1\Sq^2\pr^\infty_2$.
It follows that $\vd^1_0\KQ $ cannot be zero.
\end{proof}


\begin{remark}\label{rem:top-real-diff}
The complex realization of $\KQ$ is the real topological $K$-theory spectrum $\KO$ \cite[Lemma 2.13]{aro}.
As noted in \cite[§3.3]{grso}, 
the very effective slice filtration maps to the ``even-connectivity''-Postnikov filtration on topological spectra 
because the complex realization of a very effective motivic spectrum over $\mathbb{C}$ is connective, 
and $\Sigma^{n+(n)}$ realizes to $\Sigma^{2n}$. 
In particular, 
the complex realization of $\vs_{0}\KQ $ is the fiber of the map
\[ 
\Sq^2\pr^\infty_2 \colon H\Z\to \Sigma^2 H\Z/2
\]
of topological Eilenberg-MacLane spectra. 
Indeed, the latter is the boundary map from the zeroth to the first stage for the Postnikov filtration on $\KO$ --- see, for example, \cite[(1.3)]{fujii},
which also shows that the complex realizations of $\vd^1_0\KQ $ and $\vd_1^1\KQ $ are nonzero.
\end{remark}

For the computation of the $E^2$-page of the very effective slice spectral sequence of $\KQ$, it is crucial to determine the first differential $\vd^1_0 \KQ$ (Corollary~\ref{cor:vd-KQ}) on the coefficients of 
$\vs_0 \KQ$ (Theorem~\ref{thm:pi-v}).

\begin{lemma}\label{lem:vd-KQ-on-pi}
Let $S$ be the spectrum of a field or the ring of integers.
The map 
\[ \pi_{s-(w)}\vd^1_0\KQ \colon \pi_{s-(w)}\vs_{0}\KQ \to \pi_{s-(w)}\Sigma^{3,1}\MZ/2\iso h^{w-s+3,w+1}\]
is zero for all $s<2$ and for all $s\equiv 0,1 \bmod 4$, as well as for all $s$ if $S$ is the spectrum of a field of characteristic $2$.
For $S$ not the spectrum of a field of characteristic $2$, if $s\geq 2$ and $s\equiv 2,3 \bmod 4$, the composition
\[ \pi_{s-(w)}\f_1\vs_0\KQ \to \pi_{s-(w)}\vs_0\KQ \xrightarrow{\pi_{s-(w)}\vd^1_0\KQ} h^{w-s+3,w+1}\]
coincides with the composition
\[ \pi_{s-(w)}\f_1\vs_0\KQ \iso h^{w-s+1,w+1}\xrightarrow{\phi} h^{w-s+3,w+1}\iso \pi_{s-(w)}\Sigma^{3,1}\MZ/2\]
where $\phi(\tau^{s}x)=\tau^{s-2}\rho^2x$ for all $\tau^sx\in h^{w-s+1,w+1}$.
\end{lemma}

\begin{proof}
Note that the target group $\pi_{s-(\star)}\Sigma^{3,1}\MZ/2\cong \mathbf{h}^{\star+1-(s-2),\star+1}$ is zero for $s<2$. 
If $S$ is the spectrum of a field of characteristic $2$, the target of the map in question is zero unless $s=2$. 
In this case the source is the $2$-divisible group $\pi_{2-(w)}\vs_0\KQ\iso H^{w-2,w}$ mapping trivially to the target $h^{w+1,w+1}$.
Hence we may assume $s>1$ and exclude fields of characteristic $2$. 
By the proof of \Cref{lem:unique-v-operation}, the differential $\vd^1_0\KQ$ uniquely corresponds to a map $\f_1\vs_0\KQ\to \Sigma^{3,1}\MZ/2$ factoring as 
\[ \f_1\vs_0\KQ \to \s_1\vs_0(\KQ)\simeq \Sigma^{1,1}\MZ/2\oplus \Sigma^{3,1}\MZ/2 \xrightarrow{\pr} \Sigma^{3,1}\MZ/2\]
where the last map is the projection onto the last summand.
The first slice differential for $\vs_0\KQ$ implies that for $s>1$, $\pi_{s-(w)}\f_1\vs_0\KQ$ is the subgroup
\[ \bigl\lbrace (a,b)\in h^{w-s+1,w+1}\oplus h^{w-s+3,w+1}\colon \Sq^2a=\tau b\bigr\rbrace. \]
If $s\equiv 0,1 \bmod 4$, then $\Sq^2a=0$.
Since multiplication with $\tau$ is always injective (see \Cref{lem:hZ} for the case of $\Z$), the projection above sends this subgroup to zero.
If $s\equiv 2,3(4) \bmod 4$ and $a=\tau^sx$, then $\Sq^2a=\Sq^2(\tau^sx)=\tau^{s-1}\rho^2x$.
The equation $\Sq^2a=\tau b$ then implies $b=\tau^{s-2}\rho^2x$, which gives the result.
\end{proof}



\begin{remark}
\label{remark:xi}
Having determined the first differential $\vd^1\KQ $, 
it is natural to inquire about the higher differentials $\vd^r\KQ$.
Since $\vs_3\KQ=0$, the equivalence $\vf_4\KQ\simeq \vf_3\KQ$ produces a second differential
\[ \vd^2\KQ\colon \Sigma^{4,2}\MZ\simeq \vs_2\KQ \to \Sigma\vf_3\KQ\simeq \Sigma\vf_4\KQ \to \Sigma\vs_4\KQ\simeq \Sigma^{9,4}\vs_0\KQ\]
which desuspends to the map of motivic spectra
$\vd^2\KQ\colon \MZ\to \Sigma^{5,2}\vs_0\KQ$.
Composing it with the canonical map $\vs_0\KQ\to \MZ$ produces an element $\Xi$ in the abelian group $[\MZ,\Sigma^{5,2}\MZ]$, which in $\SH(\Z)$ has $6$ elements by \Cref{lem:endo-weight2-mz}.
This element will be used in Section \ref{section:tioMWKinhK} to derive an injectivity statement.
\end{remark}

\section{The image of Milnor-Witt $K$-theory in hermitian $K$-theory}
\label{section:tioMWKinhK}

For any field $F$, 
possibly of characteristic $2$, 
there is a graded ring homomorphism 
\[
\mu_*\colon K_*^{\MW}(F)\to\GW_*^{*}(F)=\pi_{0+\bideg}\KQ
\]
arising from the unit map for hermitian $K$-theory via Morel's calculation of the $0$-line of the stable 
homotopy groups of the motivic sphere spectrum in terms of Milnor-Witt $K$-theory $K_*^{\MW}(F)$ in \cite{MorelICM06}.
Asok and Fasel \cite[Theorem 4.1.2]{KO-map} construct the map explicitly for fields of characteristic not $2$ as follows: in degree $1$, $\mu_1$ is the isomorphism $K_1^{\MW}(F) \xrightarrow{\cong} \GW_1^{1}(F)$; in degree $-1$, $\mu_{-1}(\eta) = \eta \in \GW_{-1}^{-1}(F)$; and in general, $\mu_*$ is determined by the multiplicative structure of Milnor–Witt $K$-theory. By \cite[Lemma 4.1.6, Theorems 4.1.2, 4.3.1]{KO-map}, $\mu_*$ is an isomorphism in degrees $\le 3$. In this section, we employ the very effective slice spectral sequence to provide a short proof of this fact and to extend the discussion of $\mu_*$ to degrees four and five.

\begin{theorem}
\label{thm:symb-iso}
Let $F$ be a field.
For $n\leq 5$ the unit map $\unit_F\rightarrow\KQ_F$ induces an exact sequence 
\[
0
\to
K_n^{\MW}(F)
\xrightarrow{\mu_{n}} 
\GW_n^{n}(F)\to H^{n-4,n-2}(F)
\to 0
\]
reducing to an isomorphism
$
\mu_n\colon 
K_n^{\MW}(F) 
\xrightarrow{\cong}
\GW_n^{n}(F)
$
for $n<3$.
\end{theorem}
\begin{proof}
We argue using the very effective slice spectral sequences for $\unit$ and $\KQ$.
The unit map $\unit_F\to \KQ_F$ induces an isomorphism on very effective zero slices by 
\Cref{thm:unit-vs0-iso} below 
(see \cite[Theorem 3.30]{2020arXiv200506778B} for fields of characteristic not $2$).
Since the motivic sphere spectrum is very effective, the cofiber sequence \eqref{triangle} yields
\[
\pi_{0-(n)}\unit
\cong 
\pi_{0-(n)}\tilde{s}_{0}\unit
\cong
K_n^{\MW}(F)
\]
Using the identification $\pi_{s-(n)}\KQ\cong\GW_{s+n}^{n}(F)$ and \Cref{fig:e1-KQ}, the potentially nontrivial contributions to $\GW_n^{n}(F)$ for $n\leq 5$ in column $s=0$ are
\[
H^{n-4,n-2}(F)
\text{ and }
\coker\!\left(H^{n-5,n-2}(F)\xrightarrow{\vd_2}\pi_{0-(n)}\unit \cong K_n^{\MW}(F)\cong \widetilde{H}^{n,n}(F)\right)
\]
Composition with the canonical map $K_n^{\MW}(F)\to H^{n,n}(F)$ is induced by the element 
$\Xi\in [\MZ,\Sigma^{5,2}\MZ]$ in $\SH(\Z)$ (see \Cref{remark:xi}).
The symbol map $\mu_*$ coincides with the edge map in the slice spectral sequence for $\KQ$, a consequence of the multiplicative structure of the very effective spectral sequence for the motivic $E_\infty$-ring spectrum $\KQ$; cf.~\cite[Remark~4]{KO-map}, \cite{grso}, and \cite[Theorem~3.6.16]{zbMATH05870074}.
Thus, for $n\leq 5$, we obtain the short exact sequence
\[
0
\to
\coker 
(H^{n-5,n-2}(F) 
\xrightarrow{\vd_2}
K_n^{\MW}(F))
\xrightarrow{\mu_{n}} 
\GW_n^{n}(F)\to H^{n-4,n-2}(F)
\to 0
\]
from \Cref{fig:e1-KQ}.
For $n\leq 3$, both groups $H^{n-4,n-2}(F)$ and $H^{n-5,n-2}(F)$ vanish, so $\mu_n$ is an isomorphism. 
This uses $H^{i,j}(F)=0$ for $j<0$ and $H^{i,0}(F)=0$ unless $i=0$.
The case $n=2$ is essentially due to Suslin \cite[Theorem~4.1.11]{KO-map}, 
and for $n=3$ we note $H^{i,1}(F)=0$ unless $i=1$ (see \cite[Lecture~4]{zbMATH05051055}).

For $n=4$, the group $H^{-1,2}(F)$ is conjecturally trivial, while $h^{-1,2}(F)=h^{-1,2}_3(F)=0$. 
More generally, the Beilinson--Soul{\'e} conjecture predicts $H^{p,q}(F)=0$ for $p<0$ \cite[Preface]{zbMATH05051055}.
In any case, the differential $\vd_2\colon H^{-1,2}(F)\to K^{\MW}_4(F)$ vanishes for all fields.
Indeed, by \Cref{lem:endo-weight2-mz}, the generator 
\[
\partial^2_\infty\Sq^4\pr^\infty_2+\partial^3_\infty\mathsf{P}^1\pr^\infty_3\in [\MZ,\Sigma^{5,2}\MZ],
\]
which represents $\Xi$, factors through $h^{*,*}$ and $h^{*,*}_3$, and hence induces the zero map $H^{-1,2}(F)\to H^{4,4}(F)$.
Consequently $H^{-1,2}(F)\to K^{\MW}_4(F)$ lifts to $I^5(F)$, and further to $I^\infty(F)=0$ by \cite{arason-pfister}.
Thus $\vd_2=0$ for $n=4$.

For $n=5$, the composite $H^{0,3}(F)\to H^{5,5}(F)$ can be analyzed via power operations.  
If $\Char(F)\neq 2$, then $\Sq^4(\tau^3)=\tau\rho^4\in h^{4,5}(F)$ for the generator 
$\tau^3\in h^{0,3}(F)$.  
Thus 
$\Sq^4\circ \pr^\infty_2\colon H^{0,3}\to h^{4,5}$ is zero
whenever $\rho^4=0$, for instance if $\sqrt{-1}\in F$.  
If $\sqrt{-1}\notin F$, then $\pr^\infty_2\colon H^{0,3}(F)\to h^{0,3}(F)$ already vanishes, since 
\[
h^{0,3}(F)\xrightarrow{\partial^2_\infty} H^{1,3}(F)\xrightarrow{\pr^\infty_2} h^{1,3}(F)
\]
sends $\tau^3$ to $\tau^2\rho\neq 0$.  
If $\Char(F)=2$, the group $H^{0,3}(F)$ is uniquely $2$-divisible \cite{geisser-levine.ktheory}, so again $\Sq^4\circ \pr^\infty_2=0$.
Similarly, the composite
\[
H^{0,3}(F)\xrightarrow{\pr^\infty_3} h^{0,3}_3(F)\xrightarrow{\mathsf{P}^1} h^{4,5}_3(F)\xrightarrow{\partial^3_\infty} H^{5,5}(F)
\]
is zero.  Indeed, if $F$ contains a primitive cube root of unity $\zeta$, then $h^{0,3}_3(F)\cong\Z/3$ generated by $\zeta^3$.  
Since $\mathsf{P}^1(\zeta)=0\in h^{4,3}(F)=0$, the Cartan formula 
\cite[Proposition~9.7]{Voevodsky:power} forces $\mathsf{P}^1(\zeta^2)=\mathsf{P}^1(\zeta^3)=0$.  
Otherwise $h^{0,3}_3(F)=0$.  
It follows that 
\[
H^{0,3}(F)\xrightarrow{\vd_2} K^{\MW}_5(F) \to H^{5,5}(F)
\]
is the zero map.  Thus $\vd_2$ lifts to $H^{0,3}(F)\to I^6(F)$, and, as in the case $n=4$, further to $I^\infty(F)=0$.  
Concretely, if $\sqrt{-1}\in F$ of characteristic $\neq 2$, one may pass to a subfield of cohomological dimension $\leq 2$, where $h^{0,3}$ vanishes in degrees $>5$, forcing the generator of $h^{0,3}(F)$ to map to zero by naturality.  
If $\sqrt{-1}\notin F$ or $\Char(F)=2$, 
we have already seen that $\pr^\infty_2\colon H^{0,3}(F)\to h^{0,3}(F)$ is zero.  
Hence $\vd_2=0$ for $n=5$.

\end{proof}

\begin{remark}
\label{remark:vanishing}
The cokernel $H^{0,2}(F)$ of $\mu_4$ is trivial according to \cite[Remarks 3.3 (a)]{zbMATH06603527}.
\end{remark}

Of independent interest is the following statement, proven for fields of characteristic not $2$ in \cite[Theorem 3.30]{2020arXiv200506778B}, which was used in the proof of \Cref{thm:symb-iso}.

\begin{theorem}\label{thm:unit-vs0-iso}
For any scheme which is essentially smooth over a Dedekind scheme, the unit map $\unit\to \KQ$ induces an isomorphism on $\vs_0$.
\end{theorem}

\begin{proof}
By \cite[Lemma~B.1]{bachmann-hoyois} it suffices to treat the case of Dedekind domains.  
Then \cite[Prop.~B.3]{bachmann-hoyois} reduces the problem further to the case of fields.  
For fields of characteristic $\neq 2$, the unit map induces an isomorphism on $\tilde{s}_{0}$ by \cite[Theorem~3.30]{2020arXiv200506778B}.  
The proof of loc.~cit.~extends to characteristic $2$, using \Cref{thm:KQ-coeff-field} (implying that $\pi_{0+\bideg}\unit\to \pi_{0+\bideg}\kq$ is an isomorphism) together with \Cref{cor:slices-kq} (implying that $\s_{0}(\unit)\to \s_{0}(\kq)$ is an isomorphism).
 \end{proof}

\begin{corollary}
If $\mu_4$ is an isomorphism, 
see \Cref{remark:vanishing}, 
there is an exact sequence
\[
0\to \coker (K_4^{\sfM}(F)\rightarrow K_4(F))\to 
\GW_4^{5}(F)\to K_3^{\mathrm{ind}}(F)\to K_3(F)
\]
\end{corollary}
\begin{proof}
The Wood cofiber sequence stated in \Cref{thm:chn} implies there is a long exact sequence
\[
\cdots\to \GW_n^{n}(F)\rightarrow K_n(F)\to \GW_n^{n+1}(F)\to 
\GW_{n-1}^{n}(F)\rightarrow K_{n-1}(F)\to\cdots
\]
When $n=4$, 
then $(4,0)$-periodicity and Remark 7.7 in \cite{RO:slices} imply 
the isomorphisms
\[
\GW_3^{4}(F)\cong \GW_3^{0}(F)\cong K_3^{\mathrm{ind}}(F)\cong 
\coker(K_3^{\sfM}(F)\rightarrow K_3(F))
\]
Since the forgetful map $K_4^\MW(F)\to K_4(F)$
factors through $K_4^{\sfM}(F)$, 
we obtain the claimed exact sequence using $\mu_4$.
\end{proof}

\section{The hermitian $K$-groups of the integers}
\label{section:hkgoti}

Although our prior results apply more generally, 
in this section we restrict attention to the hermitian $K$-groups of the ring of integers $\Z$, 
often leaving this base implicit in the notation.  
The key input is the computation of motivic cohomology groups of $\Z$ with coefficients in $\Z$ and $\Z/2$.  
Recall that for $n>0$, the $n$-th Bernoulli number $B_n$ may be defined via the Riemann $\zeta$-function as $B_n=-n\zeta(1-n)$, or equivalently through the power series expansion
\[
\frac{t}{e^t-1} = \sum_{k=0}^\infty \frac{B_k t^k}{k!}.
\]
For $n>0$, let $u_n,v_n\in\N$ be coprime with $\frac{u_n}{v_n}=\frac{|B_{2n}|}{4n}$.  
Then $u_n$ is a product of (odd) irregular primes, while $v_n$ is even.

\begin{theorem}\label{thm:HZ}
The motivic cohomology groups $H^{s,w}(\Z)$ with integral coefficients are 
\begin{align*}
    H^{s,w}(\Z) & = 0 & s<0 \mathrm{\ or\ } w<0 \\
    H^{0,w}(\Z) & = \begin{cases} \Z & w=0 \\ 0 & w\neq 0 \end{cases} \\
    H^{1,w}(\Z) & = \begin{cases} 0 & w\leq 0 \\ \Z/2 & w=1 \\ \Z\oplus \Z/2 & w>1 \mathrm{\ odd} \\
     \Z/v_w & w>0 \mathrm{\ even} \end{cases}\\
    H^{2,w}(\Z) & = \begin{cases} 0 & w\in \{0,1,3\} \\ (0?) & 3<w \mathrm{\ odd} \\ \Z/2u_w(?) & 0<w \mathrm{\ even} \end{cases} \\
    H^{s,w}(\Z) & =\begin{cases} \Z/2 & 2<s\leq w \mathrm{\ and\ } w-s \mathrm{\ even} \\ 0 & 2<s \mathrm{\ and\ } (w<s \mathrm{\ or\ } w-s \mathrm{\ odd}) \end{cases}
\end{align*}
Here $(0?)$ denotes a finite group, conjecturally trivial, whose order is a product of irregular primes.  
If a group is followed by a question mark in parentheses, it is conjecturally cyclic, though always of the specified order.  
\end{theorem}

\begin{proof}
This follows from \cite{levine99} and \cite{Weibel:K-book} using the Milnor and Bloch-Kato conjectures on Galois cohomology; see \cite{spitzweck.him}.
\end{proof}

For motivic cohomology with coefficients in $\Z/2$, 
recall that $h^{\ast,\star}(\R)\iso \Z/2[\tau,\rho]$, 
where $\tau\in h^{0,1}(\R)$ and $\rho\in h^{1,1}(\R)$ are the unique nontrivial elements.

\begin{lemma}\label{lem:hZ}
The canonical map $h^{\ast,\star}(\Z)\to h^{\ast,\star}(\R)$ is a surjective ring homomorphism whose kernel is concentrated in $\ast=1$ and $\star>1$. 
More precisely, the kernel of $h^{1,w}(\Z)\to h^{1,w}(\R)$ is $\Z/2$ for $w>1$.
If $\tau\epsilon\in h^{1,2}(\Z)$ is the nonzero element in this kernel for $w=2$, 
and $\tau\in h^{0,1}(\Z),\rho\in h^{1,1}(\Z)$ are the usual elements, 
the ring $h^{\ast,\star}(\Z)$ has the following presentation:
\[ 
h^{\ast,\star}(\Z)\iso 
\FF_2[\tau,\rho,\tau\epsilon]/(\tau\epsilon)^2,\tau\epsilon\cdot \rho
\]
Finally, 
$\tau^{2m-1}\cdot \tau\epsilon $ is the image of a generator of the infinite cyclic summand in 
$H^{1,2m+1}(\Z)\cong \Z\oplus \Z/2$.
\end{lemma}

\begin{proof}
This follows from the localization cofiber sequence
\[ i_\ast i^! \MZ/2\to \MZ/2 \to j_\ast j^\ast \MZ/2 \]
for the closed embedding $i\colon \Spec(\FF_2)\hookrightarrow \Spec(\Z)$ 
and open complement $j\colon \Spec(\Z[\frac{1}{2}])\hookrightarrow \Spec(\Z)$, 
the purity equivalence $i^!\MZ/2\simeq \Sigma^{-2,-1}i^\ast \MZ/2$ 
in \cite[Theorem 7.2]{rs.lecture},  
and the presentations of $h^{\ast,\star}(\FF_2)\cong \Z/2$ 
in \cite{kato.milnor-char2} and 
$h^{\ast,\star}(\Z[\frac{1}{2}]) \iso 
\FF_2[\tau,\rho,\epsilon]/(\epsilon^2,\epsilon\rho)$ 
in \cite{arXiv:2311.13304}.
\end{proof}

\Cref{lem:hZ} implicitly determines the action of the Steenrod algebra on $h^{\ast,\star}(\Z)$, which in turn is controlled on $h^{\ast,\star}(\R)$ by the relation $\Sq^1(\tau)=\rho$ together with the Cartan formula.  
In particular, we obtain the following statement.  

\begin{lemma}\label{lem:steenrod-hZ}
The homomorphism $\Sq^2\colon h^{s,w}(\Z)\to h^{s+2,w+1}(\Z)$ is nonzero if and only if $s\leq w$ and $w-s\equiv 2,3 \bmod 4$.
If it is nonzero, it is an isomorphism unless $s=1$, and then its kernel has 2 elements, the nontrivial being $\tau^{w-2}\cdot \tau\epsilon$.
\end{lemma}

\begin{proof}
    This follows from \Cref{lem:hZ} by comparison along the inclusion $\Z\hookrightarrow \R$.
\end{proof}

A consequence of these results is an identification of the algebraic $K$-theory of $\Z$.

\begin{theorem}\label{thm:slice-spec-seq-KGLZ}
The slice spectral sequence for $\KGL_\Z$ collapses at the second page.
The algebraic $K$-groups of $\Z$ are given by 
\[ \pi_{s+(w)}\KGL_\Z = K_{s-w}(\Z) = 
\begin{cases} 
0 & s<w \\
H^{0,0}(\Z)=\Z & s=w \\
H^{1,\frac{s-w+1}{2}}(\Z) & s-w\equiv \pm 1\bmod 8 \\
H^{2,\frac{s-w+2}{2}}(\Z) & s-w\equiv 2 \bmod 8 \\
H^{3,\frac{s-w+3}{2}}(\Z)\bullet H^{1,\frac{s-w+1}{2}}(\Z) & s-w\equiv 3\bmod 8\\
H^{2, \frac{s-w+2}{2}}(\Z) & s-w\equiv 0 \bmod 4 \\
\Z  & s-w\equiv 5\bmod 8 \\
\ker (\pr\colon H^{2,\frac{s-w+2}{2}}(\Z)\to \Z/2) & s-w\equiv 6 \bmod 8
\end{cases}
\]
\end{theorem}

\begin{proof}
 The slices of $\KGL_\Z$ are given by $\s_n\KGL_\Z \simeq \Sigma^{n+(n)}\HZ$ for all $n\in \Z$ \cite{levine.homotopy-coniveau,bachmann.jems,bem.kglslices}.  
Bott periodicity $\Sigma^{1+(1)}\KGL\simeq \KGL$ implies that the first slice differential has the form $\HZ\to \Sigma^{3,1}\HZ$ up to suspensions, independent of the slice index.  
By \Cref{lem:endo-mz-weight1}, there is a unique nontrivial such map over $\Spec(\Z)$.  
Base change to $\Spec(\R)$ shows that this map is indeed nontrivial; see \cite{RO:slices}.  

On motivic cohomology groups, the differential is the composition  
\[
H^{s,w}(\Z)\xrightarrow{\pr^\infty_2}h^{s,w}(\Z)\xrightarrow{\Sq^2}h^{s+2,w+1}(\Z)\xrightarrow{\partial^2_\infty}H^{s+3,w+1}(\Z),
\]
which is nontrivial precisely for $0<s<w$ with $w-s\equiv 2\bmod 4$.  
In that case it is an isomorphism for $s>2$, and for $s\in\{1,2\}$ it projects onto a direct summand of order two, as provided by \Cref{lem:steenrod-hZ}.  

It follows that the $E^2$-page of the slice spectral sequence for $\KGL_\Z$ has at most two nonzero groups in each column.  
All nonzero entries lie along a curve of slope $\leq 1$ (asymptotically $\tfrac{1}{2}$), ruling out higher differentials.  
This proves the first claim, and the resulting groups may then be read off directly from the $E^2$-page.
\end{proof}

As preparation, we next determine the $E^2$-page of the very effective slice spectral sequence for $\KQ_\Z$, using the following results.

\begin{lemma}\label{lem:vd1-iso}
Let $s,w$ be natural numbers with $1<s$. 
Then $\pi_{s-(w)}\vs_0\KQ_\Z$ is trivial for $s> w+1$.
For $s\leq w+1$, the first differential homomorphism 
\[ \pi_{s-(w)}\vd^1\KQ_\Z\colon \pi_{s-(w)}\vs_0\KQ_\Z \to \pi_{s-1+(w)}\vs_1\KQ_\Z\]
is surjective whenever $s\equiv 2,3 \bmod 4$.
It is even an isomorphism except in the following cases:
If $w=s$ its kernel is $\Z/2$.
If $w=s+1$ its kernel is $H^{1,w}(\Z)$.
If $w=s+2$ its kernel is $H^{2,w}(\Z)$ if $s\equiv 3 \bmod 4$, and $\Z/u_w(?)$ if $s\equiv 2 \bmod 4$.
\end{lemma}

\begin{proof}
Since $s>1$, the slice spectral sequence for $\vs_0\KQ$ provides short exact sequences 
\begin{align*} 
0\to h^{w+1-s,w+1}\to \pi_{s-(w)}\vs_0\KQ \to H^{w-s,w}\to 0 && s\equiv 0,3 \bmod 4 \\
0\to h^{w+1-s,w+1}\!/\!\Sq^2\pr^\infty_2\! H^{w-1-s,w}\to \pi_{s-(w)}\vs_0\KQ \to H^{w-s,w}\to 0 && s\equiv 1 \bmod 4 \\
0\to h^{w+1-s,w+1}\to \pi_{s-(w)}\vs_0\KQ \to \ker(\Sq^2\pr^\infty_2\colon \! H^{w-s,w}\to h^{w+2-s,w+1})\to 0 && s\equiv 2 \bmod 4 
\end{align*}

If $s>w+1$, the outer groups vanish by \Cref{thm:HZ} and \Cref{lem:hZ}, yielding the first statement. 
If $s\equiv 2,3 \bmod 4$, the composition 
\[
h^{w+1-s,w+1} \to \pi_{s-(w)}\vs_0\KQ \to h^{w+3-s,w+1}
\] 
sends $\tau^s x$ to $\tau^{s-2}\rho^2 x$, giving a surjection, and even an isomorphism if $w\neq s$. 
Hence, for $w\neq s$, the short exact sequences split. The kernel is then $H^{w-s,w}$ if $s\equiv 3 \bmod 4$, and $\ker(\Sq^2\pr^\infty_2\colon H^{w-s,w}\to h^{w+2-s,w+1})$ if $s\equiv 2\bmod 4$, which vanishes for $w-s>2$ by \Cref{thm:HZ} and \Cref{lem:hZ}. 

For $w=s$, the sequences reduce to an isomorphism $h^{1,w+1}\cong \pi_{w-(w)}\vs_0\KQ$, and the kernel follows from \Cref{lem:hZ}. 
If $w=s+1$, we have 
\[
0 \to h^{2,w+1} \to \pi_{w-1-(w)}\vs_0\KQ \to H^{1,w} \to 0,
\] 
so the kernel is $H^{1,w}$. 
If $w=s+2$, the kernel is $H^{2,w}$ for $s\equiv 3 \bmod 4$, and 
\[
\ker(\Sq^2\pr^\infty_2\colon H^{2,w}\to h^{4,w+1}) \cong \Z/u_w(?)
\] 
for $s\equiv 2 \bmod 4$, with ambiguity only in cyclicity, not order.
\end{proof}

\begin{lemma}\label{lem:pi0V}
Let $w\in\Z$. 
The map $\pi_{0-(w)}\vs_0\KQ_\Z\to \pi_{0-(w)}\vs_0\KQ_\R\iso K^\MW_w(\R)$ is injective, and even an isomorphism for $w\leq 0$.
For $w>0$ the injection induces an isomorphism $\pi_{0-(w)}\vs_0\KQ_\Z\iso \Z$.
\end{lemma}

\begin{proof}
The inclusion $\Z \to \R$ induces isomorphisms 
\[
h^{n,n}(\Z) \cong h^{n,n}(\R) \quad \text{for all } n \in \N, \qquad 
H^{0,0}(\Z) \cong H^{0,0}(\R)
\] 
For $w \le 0$, the isomorphism follows from the slice filtration. In particular, the rank map 
\[
\pi_{0-(0)} \vs_0 \KQ_\Z \to \Z
\] 
has kernel $\Z$. For $n>0$, the map $H^{n,n}(\Z) \to H^{n,n}(\R)$ is injective, with image $\Z/2$ generated by $\{-1\}^n$, so injectivity of the corresponding slice map follows.  

Multiplication by $\eta$ induces $\pr^\infty_2\colon H^{n,n}(\Z)\to h^{n,n}(\Z)$ on slices, an isomorphism for $n>0$. Hence 
\[
\eta\colon \pi_{1-(1)}\vs_0\KQ_\Z \to \pi_{1-(0)}\vs_0\KQ_\Z
\] 
recovers the inclusion of the kernel $\Z$ above. More generally, for $w>0$, 
\[
\eta\colon \pi_{1-(w+1)}\vs_0\KQ_\Z \to \pi_{1-(w)}\vs_0\KQ_\Z
\] 
corresponds to the inclusion $2^w \Z \hookrightarrow 2^{w-1} \Z$.
\end{proof}

\begin{lemma}\label{lem:pi1V}
The group $\pi_{1-(w)}\vs_0\KQ_\Z$ is nonzero if and only if $w\leq 2$.
Moreover, there are isomorphisms 
\[ \pi_{1-(w)}\vs_0\KQ_\Z\iso \begin{cases} \Z/48 & w=2 \\\Z/2\oplus \Z/2 & w\in \{0,1\} \\ \Z/2 & w<0\end{cases}\]
\end{lemma}
\begin{proof}
    The second page (which equals the $\infty$-page) of the slice spectral sequence for $\vs_0\KQ$ gives the desired first statement. 
    It also produces for $w=1$ precisely one nontrivial group on the respective column, namely $h^{1,2}(\Z)\iso \Z/2\oplus \Z/2$.
    For $w=0$ two groups isomorphic to $\Z/2$ show up, so that $\pi_{1-(w)}\vs_0\KQ_Z$ could be a nontrivial extension. 
    However, for $w=0$ the nonzero element in the group $h^{0,1}(\Z)$ appearing in 
    \[ 0 \to h^{1,2}(\Z)/\tau h^{1,1}(\Z) \to \pi_{1-(0)}\vs_0\KQ_\Z \to h^{0,1}(\Z)\to 0\]
    lifts to the image of the topological Hopf map $\eta_{\mathrm{top}}\in \pi_{1-(0)}\vs_0\KQ$, which has order 2.
    Hence the short sequence above splits for $w=0$. 
    For $w<0$ the single nontrivial group appearing is $h^{1,2}(\Z)/\tau h^{1,1}(\Z)\iso \Z/2$.
    It remains to see that the extension 
    \[ 0 \to h^{2,3}(\Z)\to \pi_{1-(2)}\vs_0\KQ_\Z \to H^{1,2}(\Z) \to 0\]
    is the unique nontrivial one.
    For this purpose, one observes that $\pi_{1-(2)}\vs_0\Q$ is the single nonzero group in the column of the $\tilde{E}^1$-page for $\KQ$ computing $\pi_{1-(2)}\KQ$. 
    The Wood cofiber sequence provides an exact sequence
    \[ \pi_{1-(3)}\KQ\to \pi_{1-(2)}\KQ \to \pi_{1-(2)}\KGL\cong K_3(\Z) \to \pi_{0-(3)}\KQ\cong \Z\]
    where the last isomorphism follows as in \Cref{thm:KQ-coeff-field}. 
    Hence the last homomorphism displayed above, originating in a torsion group, has to be zero.
    Since $K_3(\Z)\cong \Z/48$ by \cite{lee-szczarba}, the extension for $\pi_{1-(2)}\vs_0\KQ$ has to contain an element of order $48$.
\end{proof}

\begin{lemma}\label{lem:more-vd1-iso}
The differential homomorphism 
\[ \partial^2_\infty\Sq^2 =\pi_{s-w+(w)}(\Sigma^{-1-(1)}\vd^1_1\KQ)\colon h^{s,w}\to H^{s+3,w+1}\]
is nonzero only if $w-s\equiv 2 \bmod 4$.
Under this condition, it is an isomorphism if $1<s<w$ and a surjection with kernel $\Z/2$ if $1=s<w$.
\end{lemma}

\begin{proof}
    This follows from \Cref{thm:HZ} and \Cref{lem:hZ}.
\end{proof}

The statements regarding the first page of the very effective slice spectral sequence for $\KQ_\Z$ allow us to determine its second page.
The periodicity $\Sigma^{4+(4)}\KQ\simeq \KQ$ implies that it suffices to 
identify $\pi_{s+(w)}\KQ$ for all $s\in \Z$ and $w\in \{0,1,2,3\}$.
The case $w=0$ corresponds to symmetric Grothendieck-Witt groups, the case $w=2$ to symplectic Grothendieck-Witt groups.
From the perspective of the very effective slice spectral sequence, the case $w=2$ is the simplest.

\begin{proposition}\label{prop:2ndpageweight2KQ}
The second page of the very effective slice spectral sequence 
computing $\pi_{s+(2)}\KQ_\Z$ consists of zero columns if 
$s\equiv 3\bmod 8$ and if $2\neq s\equiv 2\bmod 8$.
The columns for $s=2$, $s\equiv 1,5,6,7 \bmod 8$ and $2>s\equiv 0\bmod 4$ 
contain a single nonzero entry given by 
\begin{center}
\begin{tabular}{|c|c|}
\hline
$\Z/2$ & $2>s\equiv 1\bmod 4$, $2<s\equiv 6,7 \bmod 8$\\
$H^{1,\frac{s-1}{2}}(\Z)$ & $2<s\equiv 1 \bmod 8$ \\
$\pi_{\tfrac{s-3}{2}-(\tfrac{s-1}{2})}\vs_0\KQ_\Z= h^{2,\frac{s+1}{2}}(\Z)\bullet H^{1,\frac{s-1}{2}}(\Z)$ &  $2<s\equiv 5\bmod 8$ \\
$\Z$ & $s=2$, $2>s\equiv 0 \bmod 4$ \\
\hline
\end{tabular}
\end{center}
\end{proposition}

\begin{proof}
For $s<2$, $\pi_{s+(2)}\vs_n\KQ=0$ for $n\not\equiv 0\bmod 4$.
For $n\equiv 0 \bmod 4$ and $s<2$, one has 
\[ \pi_{s+(2)}\vs_{n}\KQ=\pi_{s+(2)}\Sigma^{n+(n)}\vs_0\KQ=\pi_{s-n+(2-n)}\vs_0\KQ\neq 0 \]
if and only if $s=n$ or $s=n+1$, giving $\pi_{0+(2-s)}\vs_0\KQ = \KMW_{s-2}$ and $\pi_{1+(3-s)}\vs_0\KQ\iso \Z/2$ as the single nonzero entry in this column.
The other cases follow from \Cref{lem:vd1-iso} and \Cref{lem:more-vd1-iso}, together with \Cref{lem:pi0V}, \Cref{thm:HZ} and \Cref{lem:hZ}.
\end{proof}

We now consider the remaining columns $2 < s \equiv 0 \bmod 4$ on the second page of the very effective slice spectral sequence computing $\pi_{s+(2)} \KQ_\Z$. Each such column forms a tower of nontrivial groups with top term 
$\KMW_{\frac{s+4}{2}}(\Z) \cong \Z$.
Its bottom term is $h^{0,\frac{s+4}{4}}$ 
if $s \equiv 0 \bmod 8$, 
coming from the same slice as the single nonzero entry in the column to its left; 
thus these groups are permanent cycles.  
If $s \equiv 4 \bmod 8$, the bottom term is 
$\pi_{\frac{s-4}{2}+(\frac{s}{2})}\vs_0 \KQ_\Z \cong h^{3,\frac{s+2}{2}} \bullet H^{2,\frac{s}{2}}$, 
and the column to its left is zero, so these too consist of permanent cycles.  
It remains to check whether any differential hits these columns and to resolve possible extensions. For this purpose, the Wood cofiber sequence relating weights $2$ and $3$ will be instrumental.

\begin{proposition}\label{prop:2ndpageweight3KQ}
The second page of the very effective slice spectral sequence for $\pi_{s+(3)}\KQ_\Z$ consists of zero columns if $1<s\equiv 1\bmod 8$, if $s\equiv 2\bmod 8$, and if $4>s\equiv 2,3 \bmod 4$.
The columns for $3<s\equiv 3,5,6,7 \bmod 8$ and $5>s\equiv 0\bmod 4$ contain a single nonzero entry given in the following table
\begin{center}
\begin{tabular}{|c|c|}
\hline
$\Z/2$ & $2>s\equiv 1\bmod 4$ \\
$H^{2,\frac{s-1}{2}}(\Z)$ & $3<s \equiv 3,7 \bmod 8$ \\
$h^{1,\frac{s-1}{2}}(\Z)$ & $3<s\equiv 5 \bmod 8$ \\
$h^{1,\frac{s-2}{2}}(\Z)$ &  $3<s\equiv 6\bmod 8$ \\
$\Z$ & $5>s\equiv 0 \bmod 4$ \\
\hline
\end{tabular}
\end{center}
\end{proposition}

\begin{proof}
    Similar to the proof of \Cref{prop:2ndpageweight2KQ}.
\end{proof}

The Wood cofiber sequence then yields, for $1 < s \equiv 1 \bmod 8$, an exact sequence
\[
0 = \pi_{s+1+(3)} \KQ \longrightarrow \pi_{s+1+(3)} \KGL = K_{s-2}(\Z) \longrightarrow \pi_{s+(2)} \KQ \longrightarrow \pi_{s+(3)} \KQ = 0
\]
and hence an isomorphism
$K_{s-2}(\Z) \cong \pi_{s+(2)} \KQ$.
A comparison of \Cref{thm:slice-spec-seq-KGLZ} and \Cref{prop:2ndpageweight2KQ} shows that all differentials originating from the $s$-th column in the spectral sequence for $\pi_{s+(2)} \KQ$ vanish when $s \equiv 1 \bmod 8$.
For $1 < s \equiv 5 \bmod 8$, the same holds, but the argument is more involved.
Again the Wood cofiber sequence gives an exact sequence
\[
K_{s-1}(\Z) \to \pi_{s+1+(2)}\KQ \to \pi_{s+1+(3)}\KQ \to K_{s-2}(\Z) \to \pi_{s+(2)}\KQ \to \pi_{s+(3)}\KQ \to K_{s-3}(\Z)
\]
in which the first homomorphism is zero, as its source has odd order by \Cref{thm:slice-spec-seq-KGLZ} and its target is $\Z/2$ by \Cref{prop:2ndpageweight2KQ}. 
Hence the cokernel of the second homomorphism $\eta\colon \pi_{s+1+(2)}\KQ \to \pi_{s+1+(3)}\KQ$ is $\Z/2$, using \Cref{prop:2ndpageweight3KQ} for identifying its target. 
The fourth homomorphism thus has kernel $\Z/2$ by the previous sentence.
To identify the kernel of the last homomorphism above, recall that $\KQ$ satisfies absolute purity by \cite[Theorem 8.4.2]{chn}.
The localization cofiber sequence 
\[
i_!i^!\KQ_\Z \;\longrightarrow\; \KQ_\Z \;\longrightarrow\; j_\ast j^\ast \KQ_\Z 
\]
for the closed embedding $ i \colon \Spec(\FF_{2}) \hookrightarrow \Spec(\Z)$
with open complement 
$j \colon \Spec(\Z[1/2]) \hookrightarrow \Spec(\Z)$
thus induces a long exact sequence
\[ \dotsm \to \pi_{s+1+(4)}\KQ_{\FF_2}\to \pi_{s+(3)}\KQ_\Z\to \pi_{s+(3)}\KQ_{\Z[1/2]} \to \pi_{s+(4)}\KQ_{\FF_2}\to \dotsm \]
Here the groups $\pi_{s+1+(4)}\KQ_{\FF_2}\iso \pi_{s+(4)}\KQ_{\FF_2}\iso 0$ vanish by \Cref{ex:kq-finite-fields}, giving an isomorphism $\pi_{s+(3)}\KQ_\Z\iso \pi_{s+(3)}\KQ_{\Z[1/2]}$.
The same localization cofiber sequence for $\KGL$ in place of $\KQ$ and the vanishing $\pi_{s+1+(4)}\KGL_{\FF_2}=0$ from \cite{quillen.finite} supplies an injection $K_{s-3}(\Z)\to K_{s-3}(\Z[1/2])$.
Hence the kernel of the last homomorphism above coincides with the kernel of the forgetful map $\pi_{s+(3)}\KQ_{\Z[1/2]}\to K_{s-3}(\Z[1/2])$.
By \cite[Theorem 1.4]{zbMATH05840069}, the source is isomorphic to ${}_{1}V_{s-4}(\Z[1/2])\iso (\Z/2)^2$ (up to groups of finite order).
\Cref{prop:2ndpageweight3KQ} shows that the expression in parentheses can be removed.
Hence the last homomorphism has the group $\Z/2$ as its kernel, because its source is the Klein four group and its target $\pi_{s+(3)}\KGL=K_{s-3}(\Z)$ is the direct sum of $\Z/2$ and a finite group of odd order.
There finally results an exact sequence
\[
0 \to \Z/2 \to K_{s-2}(\Z) \to \pi_{s+(2)}\KQ \to \Z/2 \to 0
\]
showing that the orders of these finite groups agree. 
Comparison of \Cref{thm:slice-spec-seq-KGLZ} and \Cref{prop:2ndpageweight2KQ} then implies that all differentials originating from the $s$-th column of the spectral sequence for $\pi_{s+(2)}\KQ$ vanish when $s \equiv 5 \bmod 8$.

\begin{theorem}\label{thm:inftypageweight2KQ}
The slice spectral sequence for $\pi_{s+(2)}\KQ_\Z$ collapses at its second page.
The weight $2$ hermitian $K$-groups of $\Z$ are given by 
\begin{center}
\begin{tabular}{|c|c|}
\hline
$\Z/2$ & $2>s\equiv 1\bmod 4$, $2<s\equiv 6,7 \bmod 8$\\
$H^{1,\frac{s-1}{2}}(\Z)$ & $2<s\equiv 1 \bmod 8$ \\
$0$ & $2>s\equiv 2,3 \bmod 4$, $2<s\equiv 2,3 \bmod 8$ \\
$\pi_{\tfrac{s-3}{2}-(\tfrac{s-1}{2})}\vs_0\KQ_\Z= h^{2,\frac{s+1}{2}}(\Z)\bullet H^{1,\frac{s-1}{2}}(\Z)$ & $2<s\equiv 5\bmod 8$ \\
$\Z$ & $s=2$, $5>s\equiv 0 \bmod 4$ \\
$\Z\oplus \ker (\pr\colon H^{2,\frac{s}{2}}(\Z)\to \Z/2)$ & $2<s\equiv 0 \bmod 4$ \\
\hline
\end{tabular}
\end{center}
\end{theorem}

\begin{proof}
As explained above, the very effective slice spectral sequence for $\pi_{\ast+(2)}\KQ_\Z$ collapses at its second page. 
A simpler argument for the vanishing of higher differentials in columns for $\pi_{s+(2)}\KQ$ with $s \equiv 1 \mod{8}$ is as follows. 
The only group in this column arises from $\vs_{\frac{s+3}{2}}\KQ$, which is killed by $\vs_{\frac{s+3}{2}}\Hyper$ (see Proposition~\ref{prop:veff-hyper}). 
Since the very effective slice spectral sequence is natural, and since the slice spectral sequence for $\KGL_\Z$ collapses at its second page (Theorem~\ref{thm:slice-spec-seq-KGLZ}), the same conclusion holds here. 

This argument does not extend to columns with $s \equiv 5 \mod{8}$. 
In this case the relevant groups arise from $\vs_{\frac{s+3}{2}}\KQ$ and are not killed by the hyperbolic map. 
We therefore turn to the case $2 < s \equiv 0 \mod{4}$. 
Recall that $\KQ$ satisfies absolute purity by \cite[Theorem 8.4.2]{chn}. 
Thus the localization cofiber sequence 
\[
i_!i^!\KQ_\Z \;\longrightarrow\; \KQ_\Z \;\longrightarrow\; j_\ast j^\ast \KQ_\Z 
\]
for the closed embedding $ i \colon \Spec(\FF_{2}) \hookrightarrow \Spec(\Z)$
with open complement 
$j \colon \Spec(\Z[1/2]) \hookrightarrow \Spec(\Z)$
induces a long exact sequence of hermitian $K$-groups
\[
\dotsb \;\to\; \pi_{s+1+(3)}\KQ_{\FF_2}
  \;\to\; \pi_{s+(2)}\KQ_\Z
  \;\to\; \pi_{s+(2)}\KQ_{\Z[1/2]}
  \;\to\; \pi_{s+(3)}\KQ_{\FF_2}
  \;\to\; \dotsb 
\]
If $s \equiv 0 \mod{4}$, then $\pi_{s+1+(3)}\KQ_{\FF_2}=0$ (Example~\ref{ex:kq-finite-fields}). 
Thus $\pi_{s+(2)}\KQ_\Z \to \pi_{s+(2)}\KQ_{\Z[1/2]}$ is injective for $s \equiv 0 \mod{4}$. 

By \cite[Theorem B]{berrick-karoubi}, 
\[
\pi_{4m+(2)}\KQ_{\Z[1/2]} \;\cong\; \Z \oplus A_m,
\]
where $A_m$ is finite of odd order. 
On the $E_2$-page of the very effective slice spectral sequence for $\pi_{4m+(2)}\KQ_\Z$, the only entry with odd torsion is $H^{2,2m}$. 
According to Theorem~\ref{thm:HZ}, this group has order $2u_{2m}$, where $u_{2m}$ is a product of irregular primes. 
Hence the odd-order part of $\pi_{4m+(2)}\KQ_\Z$ has size at most $u_{2m}$. 
Now consider the Wood cofiber sequence
\[
\pi_{8m+1+(3)}\KQ_{\Z}
   \;\to\; \pi_{8m+1+(3)}\KGL_\Z
   \;\to\; \pi_{8m+(2)}\KQ_\Z
   \;\to\; \pi_{8m+(3)}\KQ_\Z 
\]
The first term vanishes (Proposition~\ref{prop:2ndpageweight3KQ}), so $\pi_{8m+1+(3)}\KGL_\Z$, which has order $u_{4m}$ (Theorem~\ref{thm:slice-spec-seq-KGLZ}), injects into $\pi_{8m+(2)}\KQ_\Z$.
This gives $\pi_{8m+(2)}\KQ_\Z\iso \Z\oplus \Z/u_{2m}(?)$. Similarly, the part
\[
\pi_{8m+5+(3)}\KQ_{\Z}
   \;\to\; \pi_{8m+5+(3)}\KGL_\Z
   \;\to\; \pi_{8m+4+(2)}\KQ_\Z
   \;\to\; \pi_{8m+4+(3)}\KQ_\Z
\]
of the Wood sequence begins with the Klein four group, as explained right before \Cref{thm:inftypageweight2KQ}.  
The isomorphism $\pi_{8m+5+(3)}\KGL_\Z \cong H^{2,4m+2}$ (Theorem~\ref{thm:slice-spec-seq-KGLZ}) shows that the image of the Klein four group has exactly two elements. 
Since $\pi_{8m+4+(2)}\KQ_\Z$ contains no elements of order two, it follows that also this group is isomorphic to $\Z \oplus \Z/u_{2m}(?)$.
\end{proof}

The argument at the end of the proof of Theorem~\ref{thm:inftypageweight2KQ} shows that the image of the forgetful map 
$\pi_{4m+1+(3)}\KGL \;\longrightarrow\; \pi_{4m+(2)}\KQ$
is precisely the odd torsion summand. 
This observation determines the extensions from the $E_2$-page of the very effective slice spectral sequence for $\pi_{4m+(3)}\KQ_\Z$. 
Similar considerations then yield the complete coefficients of $\KQ_\Z$, i.e.,  all groups $\pi_{s+(w)}\KQ_\Z$, expressed in terms of the motivic cohomology of $\Z$.

\begin{theorem}\label{thm:coeff-KQZ}
    For $s<w$, we have 
    \[ \pi_{s+(w)}\KQ_\Z\cong 
    \begin{cases} 
        \Z & s\equiv 0 \mod 4 \\ 
        \Z/2 & s\equiv 1 \mod 4 \\
        0 & s\equiv 2,3 \mod 4
    \end{cases}\]
    For $s=w$, we have
    \[ \pi_{w+(w)}\KQ_\Z\cong 
    \begin{cases} 
        \Z\oplus \Z & w\equiv 0 \mod 4 \\ 
        \Z/2 & w\equiv 1 \mod 4 \\
        \Z & w\equiv 2 \mod 4 \\
        0 & w \equiv 3 \mod 4
    \end{cases}\]
     For $s=w+1$, we have
    \[ \pi_{w+1+(w)}\KQ_\Z\cong 
    \begin{cases} 
        (\Z/2)^3 & w\equiv 0 \mod 4 \\ 
        \Z/2 & w\equiv 1 \mod 4 \\
        0 & w\equiv 2 \mod 4 \\
        \Z & w \equiv 3 \mod 4
    \end{cases}\]
    In the case $s>w+1$, $\pi_{s+(w)}\KQ_\Z$ is given by
    \begin{center}
        \begin{tabular}{|l||c|c|c|c|}
        \hline
        $s$ & $w\equiv 0 \mod 4$ & $w\equiv 1 \mod 4$& $w\equiv 2\mod 4$ & $w\equiv 3\mod 4$ \\
        \hline
        \hline
        $8m+0$ & $\Z\oplus \Z/2$ & $\Z\oplus \Z/2$& $\Z \oplus A_{4m+4}$ & $\Z\oplus \Z$ \\
        \hline
        $8m+1$ & $(\Z/2)^3$ & $\Z/2 \oplus H^{2,4m+1}$ & $H^{1,4m}$ & $0$ \\
        \hline
        $8m+2$ & $\Z/2\oplus H^{2,4m+2}$ & $H^{1,4m+1}$ & $0$ & $0$ \\
        \hline
        $8m+3$ & $H^{1,4m+2}$ & $0$ & $0$ & $H^{2,4m+1}$ \\
        \hline
        $8m+4$ & $\Z$ & $\Z$ & $\Z\oplus A_{4m+2}$ &  $\Z\oplus \Z\oplus \Z/2$\\
        \hline
        $8m+5$ & $0$ & $H^{2,4m+3}$ & $\Z/2 \bullet H^{1,4m+2}$ & $(\Z/2)^2$ \\
        \hline
        $8m+6$ & $A_{4m+4}$ & $\Z$ & $\Z/2$ & $(\Z/2)^2$ \\
        \hline 
        $8m+7$ & $H^{1,4m+4}$ & $\Z/2$ & $\Z/2$ & $H^{2,4m+3}$\\
        \hline
         \end{tabular}
    \end{center}
    Here $A_s:=\ker (\pr \colon H^{2,s}\to \Z/2)$, and the extension for $\pi_{8m+5+(2)}\KQ$ is nontrivial.
\end{theorem} 

\begin{proof}
As mentioned, it suffices to treat the weights $w \in \{0,1,2,3\}$.  
The case of weight $2$ from \Cref{thm:inftypageweight2KQ} is the starting point.  
The analog of \Cref{prop:2ndpageweight2KQ} for weight $1$ provides an $\tilde{E}^2$-page in which only the columns for $\pi_{4m+(1)}\KQ$ and $\pi_{8m+1+(1)}\KQ$ can have more than one nonzero entry.  
In the latter case, these are $h^{0,4m} = \pi_{8m+1+(1)}\vs_{4m+1}\KQ$ and $H^{2,4m+1} = \pi_{8m+1+(1)}\vs_{4m+2}\KQ$.  
To prove that both groups consist of permanent cycles, 
consider the following portion of the Wood cofiber sequence
\begin{equation}\label{eq:wood-8m+1-weight1}
  \pi_{8m+2+(2)}\KGL \to \pi_{8m+1+(1)}\KQ \to \pi_{8m+1+(2)}\KQ \to \pi_{8m+1+(2)}\KGL
\end{equation}
The first homomorphism in \eqref{eq:wood-8m+1-weight1}, induced by the hyperbolic map, hits $H^{2,4m+1}$ bijectively by \Cref{prop:veff-hyper}.  
Hence its cokernel can be at most $h^{0,4m} \cong \Z/2$.  
By \Cref{prop:veff-forg}, the last homomorphism in \eqref{eq:wood-8m+1-weight1}, induced by the forgetful map, is multiplication by $2$ on $H^{1,4m}$, a finite nontrivial group of even order for $m>0$.  
Thus its kernel is at least $\Z/2$.  
It follows that the kernel is precisely $\Z/2$, and all elements of that column are permanent cycles.  
The two-term extension for $\pi_{8m+1+(1)}\KQ$ must split, as one group has odd order and the other has two elements.  

For degree reasons, all elements on columns for $\pi_{4m+(1)}\KQ$ must be permanent cycles, and the previous argument shows that the column for $\pi_{8m+(1)}\KQ$ is not hit.  
The single entry in the column for $\pi_{8m+5+(1)}\KQ$ has an odd number of elements and therefore maps trivially to the column for $\pi_{8m+4+(1)}\KQ$.  
Thus collapsing at the second page occurs for weight $1$ as well.  
One can then read off the group $\pi_{s+(1)}\KQ$ for $s \not\equiv 0,1,4 \mod 8$ directly.  
To determine the extension for $\pi_{4m+(1)}\KQ$ we use the Wood cofiber sequence.  
The cokernel of the last homomorphism in \eqref{eq:wood-8m+1-weight1} appears at the beginning of the exact sequence
\[
  0 \to \Z/2 \to \pi_{8m+(1)}\KQ \to \pi_{4m+(2)}\KQ \to \pi_{4m+(2)}\KGL \to \dotsm
\]
Similarly to the argument above, the kernel of the forgetful homomorphism $\pi_{4m+(2)}\KQ \to \pi_{4m+(2)}\KGL$ is isomorphic to $\Z$, thereby providing a splitting $\pi_{8m+(1)}\KQ \cong \Z \oplus \Z/2$.  
The forgetful homomorphism $\pi_{8m+5+(2)}\KQ \to \pi_{8m+5+(2)}\KGL$ is an isomorphism, since the extension for the latter group displayed in \Cref{thm:slice-spec-seq-KGLZ} is nontrivial.  
Hence $\pi_{8m+4+(1)}\KQ \cong \Z$ for all $m \geq 0$.  
This completes the description of weight $1$.  

Passage to weights $0$ and $3$ proceeds in an analogous fashion.  
For example, the extension for $\pi_{8m+(3)}\KQ$ is determined from the portion
\[
  \pi_{8m+1+(3)}\KGL \to \pi_{8m+(2)}\KQ \to \pi_{8m+(3)}\KQ \to \pi_{8m+(3)}\KGL
\]
of the Wood cofiber sequence, whereas the extension for $\pi_{8m+4+(3)}\KQ$ is determined from the portion
\[
  \pi_{8m+5+(4)}\KQ \to \pi_{8m+5+(4)}\KGL \to \pi_{8m+4+(3)}\KQ \to \pi_{8m+4+(4)}\KQ \to \pi_{8m+4+(4)}\KGL
\]
of the Wood cofiber sequence.  
Details are left to the reader.
\end{proof}
    
Inverting $\eta$ simplifies the ring structure on $\pi_{\ast+\bideg}\KQ_\Z$ considerably.

\begin{theorem}\label{thm:coeff-KW}
The coefficients of $\KW_\Z$ are given by 
\[
  \pi_{s+(w)}\KW_\Z \cong
  \begin{cases} 
    \Z & s \equiv 0 \mod{4} \\ 
    \Z/2 & s \equiv 1 \mod{4} \\ 
    0 & \text{otherwise} 
  \end{cases}
\]
The ring structure admits the presentation
\[
  \pi_{\ast+\bideg}\KW_\Z = \Z[\alpha^{\pm 1}, \eta^{\pm 1}, \xi]/(2\xi,\xi^2)
\]
Here $\alpha \in \pi_{4+(4)}\KW_\Z$ is a generator, and $\xi \in \pi_{1+(0)}\KW_\Z$ is the nonzero element.
\end{theorem}

\begin{proof}
Theorem~\ref{thm:slices-KW} shows that the second page of the slice spectral sequence for $\KQ_\Z$ has zero columns when $s \equiv 2,3 \mod{4}$.  
Furthermore, the single nonzero entry in a column with $s \equiv 1 \mod{4}$ is
\[
  \Z/2 \cong h^{1,2}(\Z) / \tau h^{1,1}(\Z)
\]
The columns for $s \equiv 0 \mod{4}$ consist of an infinite tower 
\[
  h^{0,0}(\Z), \; h^{1,1}(\Z), \; \dotsc, \; h^{n,n}(\Z), \; \dotsc
\]
Passage along the ring inclusion $\Z \to \Q$ and naturality of the slice filtration imply that all potentially nonzero slice differentials originating in a column for $s \equiv 1 \mod{4}$ vanish.  
Hence the slice spectral sequence for $\KW_\Z$ collapses at the second page.  

Comparison along the ring inclusion $\Z \to \R$ shows that the resulting extension from the columns with $s \equiv 0 \mod{4}$ yields the symmetric Witt group of $\Z$, namely the group of integers.  
The ring structure follows by comparison with the real numbers, together with the periodicities of $\KW$ and the computationally determined relations $2\xi = 0$ and $\xi^2 = 0$.
\end{proof}

\printbibliography

\end{document}